\documentclass[12pt]{amsart}
\usepackage{amsfonts,amssymb,a4wide,url,mathscinet,mathdots,bm,amsrefs,hyperref}

\newcommand{\C}{\mathbb{C}}
\newcommand{\A}{\mathbb{A}}
\newcommand{\Q}{\mathbb{Q}}
\newcommand{\Z}{\mathbb{Z}}
\newcommand{\R}{\mathbb{R}}
\newcommand{\K}{\mathbf{K}}
\newcommand{\bad}{S}
\newcommand{\dist}{d}
\newcommand{\vol}{\operatorname{vol}}
\newcommand{\sph}{\mathcal{S}_\infty}
\newcommand{\Paley}{\mathcal{P}}
\newcommand{\bdry}{\partial}
\newcommand{\Pspace}{\Paley(\aaa_{0,\C}^*)^W}
\newcommand{\PspaceR}[1]{\Paley_{#1}(\aaa_{0,\C}^*)^W}
\newcommand{\Test}{C_c^\infty(G_\infty//\K_\infty)}
\newcommand{\TestR}{C_R^\infty(G_\infty//\K_\infty)}
\newcommand{\ccg}{C_c^\infty(G(\A)^1)}
\newcommand{\modulus}{\pmb{\delta}}
\newcommand{\prms}{\mathbf{P}}
\newcommand{\maxsupp}{\mathfrak{ms}}
\renewcommand{\Re}{\operatorname{Re}}
\renewcommand{\Im}{\operatorname{Im}}
\newcommand{\noncusp}{\operatorname{ncusp}}
\newcommand{\prm}{\pmb{\lambda}}
\newcommand{\regd}{d}
\newcommand{\cfunc}{\mathbf{c}}
\newcommand{\rhosymb}{\pmb{\rho}}
\newcommand{\mult}{m}
\newcommand{\pls}{++}
\newcommand{\supp}{\operatorname{supp}}
\newcommand{\reg}{\mathfrak{d}}
\newcommand{\sa}{\R}
\newcommand{\pd}{\ge0}
\newcommand{\unrd}[1]{\widehat{#1}^{\operatorname{unr}}}
\newcommand{\disc}{\operatorname{disc}}
\newcommand{\der}{\operatorname{der}}
\newcommand{\fin}{\operatorname{fin}}
\newcommand{\Ind}{\operatorname{Ind}}
\newcommand{\Levis}{\mathcal{L}}
\newcommand{\pars}{\mathcal{P}}
\newcommand{\aaa}{\mathfrak{a}}
\newcommand{\ball}{\mathcal{B}}
\newcommand{\iii}{\mathrm{i}}
\newcommand{\one}{\mathbf{e}}
\newcommand{\nt}{\operatorname{nt}}
\newcommand{\tr}{\operatorname{tr}}
\newcommand{\mst}{T_0}
\newcommand{\AF}{\mathcal{A}}
\newcommand{\dmst}{\widehat{\mst}}
\newcommand{\dmstM}{\widehat{\mst^M}}
\newcommand{\temp}{\operatorname{temp}}
\newcommand{\tempdual}{\dmst(\C)^1}
\newcommand{\herm}{\operatorname{hm}}
\newcommand{\unt}{\operatorname{unt}}
\newcommand{\zzz}{\mathfrak{z}}
\newcommand{\ntdual}{\dmst(\C)^{\nt}}
\newcommand{\hermdual}{\dmst(\C)^{\herm}}
\newcommand{\unitdual}{\dmst(\C)^{\unt}}
\newcommand{\Hecke}{\mathcal{H}}
\newcommand{\Sat}{\mathcal{S}}
\newcommand{\invreg}{\C[\dmst]^W}

\newcommand{\bs}{\backslash}
\newcommand{\cl}{\operatorname{cl}}
\newcommand{\rest}{|}
\newcommand{\resid}{\operatorname{res}}
\newcommand{\rk}{\mathbf{r}}
\newcommand{\dm}{\mathbf{d}}

\newcommand{\cusp}{\operatorname{cusp}}
\newcommand{\GL}{\operatorname{GL}}
\newcommand{\SL}{\operatorname{SL}}
\newcommand{\pl}{\operatorname{pl}}
\newcommand{\abs}[1]{\left|{#1}\right|}
\newcommand{\norm}[1]{\lVert#1\rVert}
\newcommand{\mxml}{\mathcal{M}}
\newcommand{\ortho}{\mathcal{OB}}

\newcommand{\sprod}[2]{\left\langle{#1},{#2}\right\rangle}

\newtheorem{theorem}{Theorem}[section]
\newtheorem{lemma}[theorem]{Lemma}
\newtheorem{proposition}[theorem]{Proposition}%[subsection]
\newtheorem{corollary}[theorem]{Corollary}%[subsection]
\newtheorem{definition}[theorem]{Definition}%[subsection]
\theoremstyle{remark}
\newtheorem{remark}[theorem]{Remark}%[subsection]

\begin{document}

\title[Weyl law for congruence subgroups]{On the remainder term of the {W}eyl law for congruence subgroups of Chevalley groups}
\author{Tobias Finis}
\address{Universit\"at Leipzig, Mathematisches Institut, PF 10 09 20, D-04009 Leipzig, Germany}
\email{finis@math.uni-leipzig.de}
\author{Erez Lapid}
\address{Department of Mathematics, Weizmann Institute of Science, Rehovot 7610001, Israel}
\email{erez.m.lapid@gmail.com}

%\date{\today}

\begin{abstract}
Let $X$ be a locally symmetric space defined by a simple Chevalley group $G$ and a congruence subgroup of $G(\Q)$.
In this generality, the Weyl law for $X$ was proved by Lindenstrauss--Venkatesh.
In the case where $G$ is simply connected, we sharpen their result by giving a power saving estimate for the remainder term.
\end{abstract}

\maketitle

\setcounter{tocdepth}{1}
\tableofcontents

\section{Introduction}

\subsection{}
The Weyl law, in its basic form, states that the number $N_X(T)$ of eigenfunctions of the Laplacian
on a compact $d$-dimensional Riemannian manifold $X$ with eigenvalues $\le T^2$ satisfies the asymptotic
\[
\lim_{T \to \infty} \frac{N_X (T)}{T^d} = \frac{\vol (X)}{(4\pi)^{d/2} \Gamma (\frac{d}{2} + 1)}.
\]
The definitive result for compact Riemannian manifolds is due to
H\"ormander \cite{MR0609014}. His work implies in particular that
\[
N_X (T) = \frac{\vol (X)}{(4\pi)^{d/2} \Gamma (\frac{d}{2} + 1)} T^d + O(T^{d-1})
\]
in the compact case. The order of the remainder term is optimal without further assumptions on $X$.

The problem becomes more difficult when $X$ is not compact (but still of finite volume).
An interesting class to consider is the locally symmetric spaces of non-compact type, namely $X=\Gamma\bs G/K$,
where $\Gamma$ is a lattice in a semisimple Lie group $G$ with a maximal compact subgroup $K$.
One of Selberg's main motivations for developing the trace formula was to obtain
information on the eigenvalues of the Laplacian.
Already the existence of non-zero eigenvalues is non-evident.
For $G = \SL (2, \R)$, the case of hyperbolic surfaces, Selberg \cite{MR1159119} (see also
\cite{MR1117906}*{\S39}) showed that
\[
N_{\Gamma} (T) + M_{\Gamma} (T) = \frac{\vol (X)}{4\pi} T^2 -
\frac{\kappa}{\pi} T \log\frac{2T}{e} + O (T / \log T),
\]
where $N_{\Gamma} (T) := N_{\Gamma\bs G/K} (T)$, $\kappa$ is the number of cusps of $\Gamma\bs G/K$, and
\[
M_{\Gamma} (T) = - \frac{1}{2\pi} \int_0^T \frac{\phi'}{\phi} (\frac12 + it) dt
\]
is the winding number of the determinant
$\phi (s)$ of the scattering matrix, a meromorphic function of a complex variable $s$ that is holomorphic and has absolute value $1$ on the line $\Re s = \frac12$.
The summand $M_{\Gamma} (T)$ can be interpreted as the
contribution of the part of the continuous spectrum with Laplace eigenvalue $\le\frac14+T^2$.
It is not difficult to see that the difference between $M_{\Gamma} (T)$ and the number of poles of $\phi$ with imaginary part between $0$ and $T$, counted with their multiplicities, is $O(T)$ (and Selberg refined this to $c T + O (T / \log T)$ for a constant $c \ge 0$). Selberg also showed that
\[
\liminf_{T \to \infty} \frac{M_{\Gamma} (T)}{T \log T}
\ge \frac{\kappa}{\pi}.
\]
If $\Gamma$ is a congruence subgroup of the modular group $\SL(2,\Z)$,
then the scattering determinant $\phi(s)$ is given in terms of Dirichlet $L$-functions,
and the classical results of Riemann and von Mangoldt on the number of zeros of such an $L$-function in a strip yield the asymptotic
\[
M_{\Gamma} (T) = \frac{\kappa}{\pi} T \log T + O (T).
\]
Therefore, in this case we have
\[
N_{\Gamma} (T) = \frac{\vol (X)}{4\pi} T^2 -
\frac{2\kappa}{\pi} T \log T + O (T),
\]
and it is possible to refine the remainder term $O(T)$ here to
$c' T + O (T / \log T)$ for a suitable constant $c'$ depending on $\Gamma$.

This result led Selberg to speculate whether in general $M_{\Gamma} (T) = o(T^2)$, which would imply
the Weyl law for the discrete spectrum for an arbitrary lattice
$\Gamma$ in $\SL(2,\R)$.
However, the subsequent work of Phillips and Sarnak on the dissolution of cusp forms under deformation
of lattices led Sarnak to conjecture
that the opposite extreme holds, namely that except for
the Teichm\"uller space of the once punctured torus, the discrete spectrum of
a generic non-uniform lattice in $\SL(2,\R)$ is finite
(see \cite{MR1997348} and the references therein).

It is well known that irreducible lattices in all other semisimple Lie groups do not form continuous families.
It was conjectured by Sarnak \cite{MR1159118} that the Weyl law holds in complete generality in the case of congruence subgroups.
The appropriate form of the trace formula for congruence subgroups in arbitrary rank was developed by Arthur in the adelic language.
It is technically much more complicated than in the rank one case.

\subsection{}
An important part of the discrete spectrum is the cuspidal spectrum defined by the vanishing of all constant terms with respect to proper parabolic subgroups.
In rank one, all but a finite part of the discrete spectrum is cuspidal, but this is not true in general.
The general upper bound
\[
\limsup_{T \to \infty} \frac{N_{X,\cusp} (T)}{T^d}  \le \frac{\vol (X)}{(4\pi)^{d/2} \Gamma (\frac{d}{2} + 1)}
\]
on the cuspidal spectrum was proven by Donnelly \cite{MR664496} (it can also be obtained from a basic analysis of Arthur's trace formula).
In \cite{MR2306657}, Lindenstrauss and Venkatesh made a breakthrough by showing that the cuspidal spectrum
for congruence subgroups obeys the Weyl law.\footnote{They state their
result for quotients $\Gamma \bs G (\R) / K$, where $G$ is a split adjoint semisimple group over $\Q$ and $\Gamma$ a congruence subgroup of
$G (\Q)$, but their method is completely general.}
(Some special cases had been proved earlier -- see \cites{MR1204788, MR1823867, MR2276771}.)
Their argument uses arithmeticity in an essential way. It is based on the crucial fact that the
spectral parameters of Eisenstein series at different places are not independent of one another, but
satisfy certain simple relations.
Taking this into account, Lindenstrauss and Venkatesh constructed from any single non-trivial Hecke operator a family of adelic test functions
that act trivially on all Eisenstein series, and hence effectively only see the cuspidal spectrum.
In contrast to previous instances of the simple trace formula these test functions are not factorizable,
and as a family, do not entail a loss of a positive proportion of the cuspidal spectrum.

The result of \cite{MR2306657} does not provide a bound on the remainder term in the Weyl law for the cuspidal spectrum.
In fact, it seems that as it stands, the method is short of giving a good error term, since it uses a single Hecke operator.
The purpose of the current paper is to push the basic idea of Lindenstrauss--Venkatesh further and to bound
the error term in the cuspidal Weyl law by $O(T^{d-\delta})$ for some $\delta>0$.
For simplicity, we work in the setting of simple Chevalley groups $G$ defined over $\Q$, mainly because the necessary estimates
for the geometric side of the trace formula have been worked out only in this case.
Then, we have the following (see Theorem \ref{TheoremWeyl} and Corollary \ref{cor: stdweylaw}).\footnote{In the
non simply connected case, our result is for manifolds which may be non-connected.}

\begin{theorem} \label{thm: basicweylaw}
Let $G$ be a simply connected, simple Chevalley group. Then there exists $\delta>0$ such that for any congruence subgroup
$\Gamma$ of $G(\Z)$ we have
\[
N_{X,\cusp}(T) = \frac{\vol (X)}{(4\pi)^{d/2} \Gamma (\frac{d}{2} + 1)} T^d + O_{\Gamma}(T^{d-\delta}),\ \ T\ge1
\]
where $X=\Gamma\bs G(\R)/K$.
\end{theorem}

Our method uses Hecke operators as well, but in a slightly different way. Namely, instead of \emph{annulling} the contribution
of the non-cuspidal spectrum, we \emph{amplify} it in order to show its negligibility (by a factor of $O(T^{-\delta})$).
This is somewhat analogous to the situation in Selberg's sieve (e.g., \cite{MR2200366}) and requires the use of
$T^c$ many Hecke operators for a suitable $c>0$.
The argument crucially relies on a simple positivity property of Arthur's trace formula.
Apart from the basic work of Arthur's first papers on the trace formula, we use the estimates on its geometric side of \cite{1905.09078},
which are based on \cites{MR2801400,MR3534542}.
The exact power saving that we get can in principle be computed, but since it results from an application of the
Stone--Weierstrass theorem, we expect
it to be quite poor. It would be interesting to analyze this question more carefully.
The final result (Theorem \ref{TheoremWeyl}) is actually more general than the Weyl law for the Laplacian, since we treat the entire commutative algebra of all invariant differential operators on $G (\R) / K$ simultaneously, following the method of \cite{MR532745}.

More generally, our proof gives a main term for the trace of an arbitrary Hecke operator $\tau$ on the cuspidal spectrum,
with a remainder term of order $O_\tau(T^{d-\delta})$
where the implicit constant is the $L^1$-norm of $\tau$ times a logarithmic factor depending on the support
of $\tau$ (see Theorem \ref{TheoremWeylHecke} for the precise statement).
As in \cite{MR3675175}, these estimates have applications to the conjectures of Katz--Sarnak on low-lying zeros of $L$-functions
for the family of cuspidal automorphic representations of $G$ of a bounded level that are spherical at infinity.

\subsection{}
An alterative strategy (see \cites{MR2541128,MR3711830,1905.09078}) is to use Arthur's fine spectral expansion and to analyze the analytic
properties of intertwining operators (i.e., in this case, of their global normalizing factors) in more depth.

A first step in this direction is to establish polynomial upper bounds for the discrete spectrum.
This is the trace-class conjecture, which was solved by Werner M\"uller some time ago \cite{MR1025165} (see also \cites{MR1470422, MR1622604}).
A refinement of this statement is the absolute convergence of the spectral side of Arthur's trace formula, which was established in
\cite{MR2811597}.

In \cite{MR3711830}, we formulated a precise analytic conjecture on intertwining operators,
which would imply the Weyl law for the discrete spectrum with an error term of
$O(T^{d-1})$ as in the cocompact case (with an extra logarithmic factor in the case of groups of type $A_1$ and $A_2$).
It also implies that the non-cuspidal discrete spectrum is bounded by $O(T^{d-2})$ \cite{1905.09078}.
Using the work of Langlands on the relation between intertwining operators and automorphic
$L$-functions \cite{MR0419366}, the conjecture can be formulated in terms of the latter.
This conjecture is known to hold for the general linear groups, where the pertinent $L$-functions are the Rankin--Selberg convolutions,
whose analytic properties are well-understood by the work of Jacquet--Piatetski-Shapiro--Shalika and others \cite{MR2541128}.
The conjecture is also known for quasi-split classical groups using Arthur's work and for
the split exceptional group $G_2$, by Shahidi's work on the symmetric cube $L$-function for $\GL_2$ (see \cite{MR3711830}).
In general however, the required information on the behavior of the automorphic $L$-functions is not available.

In contrast, the method of the current paper does not give an upper bound on the non-cuspidal discrete spectrum,
which remains an interesting open problem in general.

\subsection{}
To close this introduction, we give a quick summary of the individual sections of this paper.
In \S \ref{SectionArthur}, we give a summary of the first stage of the spectral expansion of Arthur's trace formula.
In \S \ref{SectionPadic}, we construct for each prime $p$, Hecke operators at $p$ suitable for the task of emphasizing the non-cuspidal contribution. The main result of this section is Proposition \ref{prop: tech}, which is the technical heart of the paper.
In \S \ref{SectionArch}, we collect some mostly standard facts about spherical archimedean test functions,
the associated Paley--Wiener theorem and the spherical Plancherel measure.
Finally, in \S \ref{SectionMainResults} we prove our main results Theorem \ref{TheoremWeyl} and Theorem \ref{TheoremWeylHecke}.
Proposition \ref{prop: mubnd} contains the key part of the argument.

\section{Review of Arthur's trace formula} \label{SectionArthur}
In this section we recall the relevant facts from the basic theory of Arthur's (non-invariant) trace formula
and set some notation which will be used throughout. We will freely use standard results from the textbook \cite{MR1361168} (and by extension,
\cite{MR0579181}) as well as from Arthur's fundamental papers \cites{MR518111, MR558260, MR681737} (see also the first
part of \cite{MR2192011}).
On the other hand, we will not go into Arthur's fine spectral and geometric expansions (let alone his more advanced
theory of the trace formula) since we will not use them in the sequel.
In fact, on the geometric side, we will only use the estimates of
\cites{MR2801400,MR3534542,1905.09078} (see \S\ref{sec: Finis-Matz} below).

For this section, let $G$ be an arbitrary reductive group defined over $\Q$.

Here and henceforward, $X \ll Y$ means that there is a constant $C$ (implicitly depending on the group $G$) such that $X\le CY$.
If $C$ depends on additional parameters, we will emphasize it by writing $X\ll_{a,b}Y$.

\subsection{}
As usual, we denote by $\A=\R\times\A_{\fin}$ the ring of adeles.
Let $\zzz$ be the center of the universal enveloping algebra of the (complexified) Lie algebra of $G_\infty:=G(\R)$.
Fix a maximal compact subgroup $\K=\prod_{p\le\infty}\K_p$ of $G(\A)$
such that the factors $\K_p$ are special for all $p$ and hyperspecial for almost all $p$.

Fix a Haar measure on $G(\A)$.
Let $\ccg$ be the $*$-algebra (under convolution, with $f^*(g)=\overline{f(g^{-1})}$)
of smooth, complex-valued, compactly supported functions on $G(\A)^1$.
(By definition, smoothness implies that the function is right-invariant under a suitable open subgroup of $G(\A_{\fin})$.)
The right regular representation gives rise to a $*$-representation $R$ of
$\ccg$ on the Hilbert space $L^2(G(\Q)\bs G(\A)^1)$.
Explicitly, for any $f\in \ccg$, $R(f)$ is the operator
\[
R(f)\varphi(x)=\int_{G(\A)^1}f(g)\varphi(xg)\ dg,\ \ \varphi\in L^2(G(\Q)\bs G(\A)^1),\  x\in G(\Q)\bs G(\A)^1
\]
which is an integral operator with kernel
\[
K_f(x,y)=\sum_{\gamma\in G(\Q)}f(x^{-1}\gamma y),\ \ x,y\in G(\Q)\bs G(\A)^1.
\]

Let $P=M\ltimes U$ be a parabolic subgroup of $G$ defined over $\Q$ with unipotent radical $U$
and Levi subgroup $M$ defined over $\Q$. Denote by $\modulus_P$ the modulus function of $P(\A)$ and
by $A_M$ the identity connected component of $T_M(\R)$,
where $T_M$ is the split part of the center of $M$. Thus, $A_MM(\Q)\bs M(\A)$ is of finite volume.
Denote by $L^2_{\disc}(A_MM(\Q)\bs M(\A))$ the discrete part of $L^2(A_MM(\Q)\bs M(\A))$.
The space of square-integrable automorphic forms on $A_MM(\Q)\bs M(\A)$ is dense in $L^2_{\disc}(A_MM(\Q)\bs M(\A))$.
Let
\begin{multline*}
\AF^2_P=\{\varphi:U(\A)M(\Q)\bs G(\A)\rightarrow\C\text{ smooth and $\zzz$-finite}:\\
\varphi(ag)=\modulus_P(a)^{\frac12}\varphi(g)\ \ \forall a\in A_M,\ g\in G(\A),
\int_{A_MM(\Q)U(\A)\bs G(\A)}\abs{\varphi(g)}^2\ dg<\infty\}
\end{multline*}
and let $L^2_P$ be the Hilbert completion of $\AF^2_P$.
We may identify $L^2_P$ with the (normalized) parabolic induction $\Ind_{P(\A)}^{G(\A)}L^2_{\disc}(A_MM(\Q)\bs M(\A))$.
Let $\Pi_2(M)$ be the set of equivalence classes of irreducible representations of $M(\A)$ that
occur discretely in $L^2(A_MM(\Q)\bs M(\A))$. (The central character of any $\pi\in\Pi_2(M)$ is trivial on $A_M$.)
For any $\pi\in\Pi_2(M)$ let $\AF_{P,\pi}^2$ be subspace of $\AF_P^2$ consisting of the functions
such that for all $x\in G(\A)$ the function $m\in M(\A)\mapsto\modulus_P(m)^{-\frac12}\varphi(mx)$ belongs to the $\pi$-isotypic component
of $L^2_{\disc}(A_MM(\Q)\bs M(\A))$. Let $L^2_{P,\pi}$ be the closure of $\AF_{P,\pi}^2$ in $L^2_P$. Thus,
\[
\AF^2_P=\oplus_{\pi\in\Pi_2(M)}\AF^2_{P,\pi}
\]
and
\[
L^2_P=\widehat\oplus_{\pi\in\Pi_2(M)}L^2_{P,\pi}.
\]
For any $\pi\in\Pi_2(M)$ we fix an orthonormal basis $\ortho_P(\pi)$ of $\AF_{P,\pi}^2$.

Let $X^*(M)$ be the lattice of rational characters of $M$ and let $\aaa_M^*=X^*(M)\otimes\R$ be
the real vector space generated by $X^*(M)$. We also write $\aaa_{M,\C}^*=X^*(M)\otimes\C$.
The restriction map $X^*(M)\rightarrow X^*(T_M)$ identifies $\aaa_M^*$ with $X^*(T_M)\otimes\R$.
The dual space of $\aaa_M^*$ will be denoted by $\aaa_M$. Thus, $\aaa_M=X_*(T_M)\otimes\R$
where $X_*(T_M)$ is the lattice of co-characters of $T_M$. We denote the canonical pairing
on $\aaa_M^*\times\aaa_M$ by $\sprod{\cdot}{\cdot}$.
Let $H_M:M(\A)\rightarrow\aaa_M$ be the homomorphism characterized by
\[
e^{\sprod\chi{H_M(m)}}=\abs{\chi(m)},\ \ \chi\in X^*(M).
\]
The kernel of $H_M$ is $M(\A)^1$ and the restriction of $H_M$ to $A_M$ is an isomorphism.

Fix a maximal $\Q$-split torus $\mst$ of $G$ that is in a good position with respect to $\K$,
i.e., for every $p$ the group $\K_p$ is the stabilizer of a special point in the apartment associated to $T_0$.
Denote by $W$ the Weyl group $N_{G(\Q)}(\mst)/\mst$.
It acts on $\aaa_0:=\aaa_{\mst}$ and $\aaa_0^*$.
We may identify $\aaa_0$ with the Lie algebra of $\mst(\R)$.
We fix a $W$-invariant Euclidean structure on $\aaa_0$.
(Of course, if $G$ is semisimple, then the Killing form is a $W$-invariant inner product on $\aaa_0$.)

Let $\pars$ be the finite set of parabolic subgroups defined over $\Q$ and
containing $\mst$. Each $P\in\pars$ admits a unique Levi decomposition $P=M\ltimes U$ with $M\supset\mst$
(necessarily defined $\Q$).
We write $\Levis$ for the set of all such Levi subgroups as we vary $P\in\pars$.
In other words, $\Levis$ is the set of centralizers of subtori of $\mst$ in $G$.
We say that $P,Q\in\pars$ are associate, denoted $P\sim Q$, if their Levi parts are conjugate in $G(\Q)$.
For every $M\in\Levis$ we identify $\aaa_M$ as a subspace of $\aaa_0$.
We have orthogonal decompositions
\[
\aaa_0=\aaa_0^M\oplus\aaa_M,\ \ \aaa_0^*=(\aaa_0^M)^*\oplus\aaa_M^*
\]
where
\[
\aaa_0^M=X_*(\mst\cap M^{\der})\otimes\R=X_*(\mst/T_M)\otimes\R
\]
and
\[
(\aaa_0^M)^*=X^*(\mst\cap M^{\der})\otimes\R=X^*(\mst/T_M)\otimes\R.
\]
Here, $M^{\der}$ is the derived group of $M$.
For any $M\in\Levis$ we write $\pars(M)$ for the set of $P\in\pars$ with Levi part $M$
and let $W(M)=N_{G(F)}(M)/M$, which can be identified with a subgroup of $W$.
For any $P\in\pars(M)$ let $\Delta_P\subset X^*(T_M)\subset\aaa_M^*$ be the corresponding set of simple roots.
(These are the non-zero projections to $\aaa_M^*$ of the simple roots with respect to any minimal parabolic subgroup in $\pars$
contained in $P$.)
We extend $H_M$ to a left $U_P(\A)$ and right $\K$-invariant map $H_P:P(\A)\rightarrow\aaa_M$.
For any $\lambda\in\aaa_{M,\C}^*$ and a function $\varphi$ on $G(\A)$ let $\varphi_\lambda$ be the function
\[
\varphi_\lambda(x)=\varphi(x)e^{\sprod{\lambda}{H_P(x)}}.
\]
This gives rise to the family of representations $I_P(\lambda)$ of $G(\A)$ on $L^2_P$ given by
\[
(I_P(g,\lambda)\varphi)_\lambda(x)=\varphi_\lambda(xg),\ \ x,g\in G(\A).
\]
For any $f\in \ccg$ we write
\[
I_P(f,\lambda)=\int_{G(\A)^1}f(g)I_P(g,\lambda)\ dg,
\]
a bounded operator on $L^2_P$.

Define a height $\norm{\cdot}$ on $G(\A)$ and heights $\norm{\cdot}_p$ on $G(\Q_p)$, $p\le\infty$,
as in \cite{MR518111}. We have $\norm{x}=\prod_{p\le\infty}\norm{x_p}_p$.

\subsection{}
For any $\varphi\in\AF_P^2$ the Eisenstein series
\[
E_P(\varphi,\lambda)=\sum_{\gamma\in P(\Q)\bs G(\Q)}\varphi_\lambda(\gamma g),\ \ g\in G(\A),
\]
which converges for $\Re\sprod{\lambda}{\alpha^\vee}\gg1$ for all $\alpha\in\Delta_P$,
admits a meromorphic continuation to $\aaa_{M,\C}^*$ and is analytic for $\lambda\in\iii\aaa_M^*$ (see \cite{MR2402686} for the non-$\K$-finite case).

By the theory of Eisenstein series we have a spectral expansion
\[
K_f(x,y)=\sum_{P\in\pars/\sim}K_{f,P}(x,y)
\]
over associate classes of parabolic subgroups, where for any $P\in\pars(M)$
\[
K_{f,P}(x,y)=\abs{W(M)}^{-1}\sum_{\pi\in\Pi_2(M)}\int_{\iii(\aaa_M^G)^*}
\sum_{\varphi\in\ortho_P(\pi)}E_P(x,I_P(f,\pi,\lambda)\varphi,\lambda)
\overline{E_P(y,\varphi,\lambda)}\ d\lambda,
\]
and $I_P(f,\pi,\lambda)$ is the restriction of $I_P(f,\lambda)$ to $L^2_{P,\pi}$.

We note that for any $f_1,f_2\in \ccg$ we have
\begin{multline*}
K_{f_1*f_2^*,P}(x,y)=\\\abs{W(M)}^{-1}\sum_{\pi\in\Pi_2(M)}\int_{\iii(\aaa_M^G)^*}
\sum_{\varphi\in\ortho_P(\pi)}E_P(x,I_P(f_1,\pi,\lambda)\varphi,\lambda)
\overline{E_P(y,I_P(f_2,\pi,\lambda)\varphi,\lambda)}\ d\lambda.
\end{multline*}

\subsection{}
Fix a minimal parabolic subgroup $P_0=\mst\ltimes U_0\in\pars(\mst)$ and let $\Delta_0\subset X^*(\mst)$ be the corresponding
set of simple roots.
For any $T\in\aaa_0$ let $\regd(T)=\min_{\alpha\in\Delta_0}\sprod{\alpha}T$
and let $\Lambda^T$ be Arthur's truncation operator \cite{MR558260} with respect to $P_0$.
It takes functions of uniform moderate growth to rapidly decreasing functions provided that $\regd(T)>C_0$ for
some constant $C_0$ depending only on $G$. Under this condition, $\Lambda^T$ also defines an orthogonal projection on $L^2(G(\Q)\bs G(\A)^1)$.
We will write
\[
\aaa_0^{\pls}=\{T\in\aaa_0:\regd(T)>C_0\}.
\]

Let $\le$ be the partial order on $\aaa_0^{\pls}$ defined by $T_1\le T_2$ if $T_2-T_1$
is a linear combination of simple co-roots with non-negative coefficients.
Suppose that $T_1\le T_2$. Then by \cite{MR558260}*{Lemma 1.1} we have
$\Lambda^{T_2}\Lambda^{T_1}=\Lambda^{T_1}$. Thus, for any $\varphi$ of uniform moderate growth we have
\[
\norm{\Lambda^{T_1}\varphi}_2^2=(\Lambda^{T_1}\varphi,\varphi)=(\Lambda^{T_2}\Lambda^{T_1}\varphi,\varphi)=
(\Lambda^{T_1}\varphi,\Lambda^{T_2}\varphi)\le\norm{\Lambda^{T_1}\varphi}_2\norm{\Lambda^{T_2}\varphi}_2.
\]
Hence,
\begin{equation} \label{eq: mntrnc}
\norm{\Lambda^{T_1}\varphi}_2\le\norm{\Lambda^{T_2}\varphi}_2.
\end{equation}

Let $T\in\aaa_0^{\pls}$ and consider the operator $\Lambda^T\circ R(f)\circ\Lambda^T$ on $L^2(G(\Q)\bs G(\A)^1)$.
This is a trace class integral operator whose kernel is given by
\[
K^T_f(x,y)=\Lambda_{G\times G}^{(T,T)}K_f(x,y),
\]
where $\Lambda^{(T,T)}_{G\times G}$ denotes truncation in both variables.
Let
\begin{equation} \label{DefinitionJT}
J^T(f)=\tr (\Lambda^T\circ R(f)\circ\Lambda^T)=\int_{G(\Q)\bs G(\A)^1}K^T_f(x,x)\ dx.
\end{equation}
Note that
\[
J^T(f)=\tr (R(f)\circ\Lambda^T)=\int_{G(\Q)\bs G(\A)^1}\Lambda_y^T K_f(x,y)\big|_{y=x}\ dx,
\]
which is how this distribution is defined in \cite{MR558260}.

On the other hand, Arthur introduced in \cite{MR518111} the modified kernel $k^T_f$ and its integral
\begin{equation} \label{DefinitionJTGeometric}
\mathbb{J}^T(f)=\int_{G(\Q)\bs G(\A)^1} k^T_f(x,x)\ dx,
\end{equation}
which by \cite{MR3534542}*{Theorem 5.1} is absolutely convergent and a polynomial function of the parameter $T$ for all $T \in \aaa_0^{\pls}$.

By a basic result of Arthur \cite{MR681737}*{Proposition 2.2} there exists a constant $C_1>0$ depending only on $G$ we have
\begin{equation} \label{eq: polyJ}
J^T(f) = \mathbb{J}^T(f)
\quad\text{for}\quad \regd(T) > C_1 \quad \text{and} \quad \supp f\subset \{x\in G(\A):\log\norm{x} < C_1^{-1} \regd(T) \}.
\end{equation}

Spectrally expanding $K^T_f$, we may write
\[
K^T_f(x,y)=\sum_{P\in\pars/\sim}\Lambda_{G\times G}^{(T,T)}K_{f,P}(x,y),
\]
where
\begin{multline*}
\Lambda_{G\times G}^{(T,T)}K_{f,P}(x,y)=\\\abs{W(M)}^{-1}
\sum_{\pi\in\Pi_2(M)}\int_{\iii(\aaa_M^G)^*}
\sum_{\varphi\in\ortho_P(\pi)}\Lambda^TE_P(x,I_P(f,\pi,\lambda)\varphi,\lambda)
\overline{\Lambda^TE_P(y,\varphi,\lambda)}\ d\lambda
\end{multline*}
for any $P\in\pars(M)$. Thus,
\[
J^T(f)=\sum_{P\in\pars/\sim}J_P^T(f)
\]
with
\begin{multline*}
J_P^T(f)=\int_{G(\Q)\bs G(\A)^1}\Lambda_{G\times G}^{(T,T)}K_{f,P}(x,x)\ dx=
\\\abs{W(M)}^{-1}
\sum_{\pi\in\Pi_2(M)}\int_{\iii(\aaa_M^G)^*}
\sum_{\varphi\in\ortho_P(\pi)}\int_{G(\Q)\bs G(\A)^1}\Lambda^TE_P(x,I_P(f,\pi,\lambda)\varphi,\lambda)
\overline{\Lambda^TE_P(x,\varphi,\lambda)}\ dx\ d\lambda.
\end{multline*}
In particular, we have the following crucial positivity property.
\[
J_P^T(f*f^*)=\abs{W(M)}^{-1}
\sum_{\pi\in\Pi_2(M)}\int_{\iii(\aaa_M^G)^*}
\sum_{\varphi\in\ortho_P(\pi)}\norm{\Lambda^TE_P(\cdot,I_P(f,\lambda)\varphi,\lambda)}_2^2\ d\lambda\ge0.
\]
Note that $J^T_P(f)$ is not necessarily a polynomial in $T$, even if $P=G$ and $\regd(T)$ is large.
(In general, $J^T_P(f)$ approximates a polynomial in $T$ as $\regd(T)\rightarrow\infty$
but we will not use this fact.)

Let $L^2_{\cusp}(A_G G(\Q)\bs G(\A))$ be the cuspidal part of $L^2(A_G G(\Q)\bs G(\A))$
and let $L^2_{\resid}(A_G G(\Q)\bs G(\A))$ be the orthogonal complement of $L^2_{\cusp}(A_G G(\Q)\bs G(\A))$
in $L^2(A_G G(\Q)\bs G(\A))$.
Denote by $\Pi_{\cusp}(G)$ (resp., $\Pi_{\resid}(G)$) the set of equivalence classes of irreducible representations that occur in
$L^2_{\cusp}(A_G G(\Q)\bs G(\A))$ (resp., $L^2_{\resid}(A_G G(\Q)\bs G(\A))$).
Note that in general $\Pi_{\cusp}(G)$ and $\Pi_{\resid}(G)$ are not disjoint.
For any $f\in\ccg$ let $R_{\cusp}(f)$ denote the restriction of $R(f)$ to
$L^2_{\cusp}(G(\Q)\bs G(\A)^1) \simeq L^2_{\cusp}(A_G G(\Q)\bs G(\A))$.
We will need the following result due to Wallach.

\begin{lemma}[\cite{MR733320}] \label{lem: wallach}
The local components $\pi_p$ of any $\pi\in\Pi_{\resid}(G)$ are non-tempered.
\end{lemma}

Although in [ibid.] this is technically only stated for the archimedean components,
the same proof, but easier, applies to the non-archimedean components as well.
Namely, the cuspidal exponents of any square-integrable automorphic form $\varphi$ lie in the negative obtuse Weyl chamber \cite{MR1361168}*{I.4.11}.
On the other hand, if $\varphi$ occurs in the space of $\pi\in\Pi_2(G)$, then every cuspidal exponents of $\varphi$
is an exponents of $\pi_p$ for all $p$. Hence, $\pi_p$ cannot be tempered
unless $\varphi$ is cuspidal.

\section{A non-archimedean separation lemma} \label{SectionPadic}

In this section we construct the Hecke operators (one for each prime) that will be used to amplify the non-cuspidal
part of the trace formula. The construction is elementary, using the Stone--Weierstrass theorem as the main tool.

From now on we assume that $G$ is a Chevalley group over $\Q$ of rank $\rk$ (in this section it would be sufficient to assume that $G$ is split reductive over  $\Q$).

\subsection{} \label{sec: heckebasic}
Let us first recall some basic facts and set some notation pertaining to unramified representations
and the Satake isomorphism. See \cites{MR0435301, MR1696481} for standard references.
Let $\dmst$ be the torus dual to $\mst$, considered as a torus over $\C$.
We denote by $\nu\mapsto\nu^\vee$ the resulting isomorphism $X^*(\dmst)\rightarrow X_*(\mst)$
between the lattices of rational characters of $\dmst$ and the co-characters of $\mst$.

For every prime $p$ let $G_p=G(\Q_p)$, $\K_p=G(\Z_p)$ and denote by
$\unrd{G_p}$ the unramified (admissible) spectrum of $G_p$.
For any $\lambda\in\dmst(\C)$ let $\chi_\lambda$ be the unramified
character of $\mst(\Q_p)$ such that $\chi_\lambda(\nu^\vee(p))=\nu(\lambda)$ for any $\nu\in X^*(\dmst)$.
The map $\lambda\mapsto\chi_\lambda$ defines an isomorphism of topological groups
between $\dmst(\C)$ and the group of unramified characters of
$\mst (\Q_p)$. (Note that the Levi part of $P_0$ is $T_0$ since $G$ is split.)
For any $\lambda\in\dmst(\C)$, the induced representation
$\Ind_{P_0 (\Q_p)}^{G (\Q_p)}\chi_\lambda$ admits a unique unramified irreducible subquotient
$\pi_\lambda$, and the isomorphism class of $\pi_\lambda$ depends only on $W\lambda$.
The map $\lambda\mapsto\pi_\lambda$ gives rise to a homeomorphism
\[
\dmst(\C)/W\rightarrow\unrd{G_p}.
\]
For any $\pi\in\unrd{G_p}$ we will write $\prm_\pi\in\dmst(\C)/W$ for the Frobenius--Hecke parameter of $\pi$.
For instance, if $\pi$ is the identity representation, then $\prm_\pi$ is the $W$-orbit of the element
$\rhosymb_p\in\dmst(\C)$ that corresponds to the square-root of the modulus character of $P_0 (\Q_p)$,
an unramified character of $\mst (\Q_p)$.

Let $\hermdual$ be the hermitian part of $\dmst(\C)$, namely
\[
\hermdual=\cup_{w\in W}\{t\in\dmst(\C):w(\bar t)=t^{-1}\}
\]
where $\bar t$ is the complex conjugate of $t$.
Clearly $\hermdual$ is closed in $\dmst(\C)$.
Let
\[
\tempdual=\{t\in\dmst(\C):\bar t=t^{-1}\}\subset\hermdual
\]
be the maximal compact subgroup of $\dmst(\C)$. Finally, set
\[
\ntdual=\hermdual\setminus\tempdual.
\]
The sets $\hermdual$, $\tempdual$ and $\ntdual$ are $W$-invariant.
Their quotients by $W$ parameterize the sets of hermitian, tempered and
non-tempered (unramified, irreducible) representations of $G_p$, respectively.
In particular, the set of unitarizable, unramified, irreducible representations of $G_p$
corresponds to a compact, $W$-invariant subset $\unitdual_p$ of $\hermdual$ (which of course depends on $p$).
Note that $\unitdual_p$ strictly contains $\tempdual$ (unless $G$ is trivial).

For any commutative $*$-algebra $A$ we denote by $A_{\sa}$ the real subalgebra of self-adjoint elements of $A$ and by
$A_{\pd}\subset A_{\sa}$ the convex cone generated by $x^*x$, $x\in A$.

Let $\invreg$ be the commutative $*$-algebra of $W$-invariant, regular functions on $\dmst$
where $h^*(t^{-1})=\bar h(t)=\overline{h(\bar t)}$.
As a vector space, $\invreg$ has the basis $e_\lambda=\frac1{\abs{W}}\sum_{w\in W}w\lambda$, $\lambda\in X^*(\dmst)/W$.
Note that $h^*(t)=\overline{h(t)}$ for any $h\in\invreg$ and $t\in\hermdual$.
Thus, $h(t)\in\R$ (resp., $h(t)\ge0$) for all $h\in\invreg_{\sa}$ (resp., $h\in\invreg_{\pd}$) and $t\in\hermdual$.

For any $\chi\in X^*(\mst)$ with corresponding co-character $\chi^\vee\in X_*(\dmst)$ let $p^\chi=\chi^\vee(p)\in\dmst(\C)$.
We extend the homomorphism $\chi\in X^*(\mst)\mapsto p^\chi\in\dmst(\C)$ to a surjective homomorphism
$\aaa_{0,\C}^*\rightarrow\dmst(\C)$ characterized by $p^{z\chi}=\chi^\vee(p^z)$
for any $\chi\in X^*(\mst)$ and $z\in\C$.

\subsection{} \label{sec: padicmes}
Fix the Haar measure on $G_p$ such that $\vol\K_p=1$.
Let $\Hecke_p$ be the Hecke algebra of bi-$\K_p$-invariant, compactly supported functions on $G_p$
with respect to convolution, with identity element $\one_{\K_p}$.
For any $f\in\Hecke_p$ we write $f^*(x)=\overline{f(x^{-1})}$.
We denote by
\[
\Sat=\Sat_p:\Hecke_p\rightarrow\invreg
\]
the Satake transform. It is an isomorphism of commutative $*$-algebras which is characterized by the property
\[
\hat f(\pi) = (\Sat_p f) (p^\lambda)\text{ for any }\pi=\pi_\lambda\in\unrd{G_p}
\]
where $\hat f(\pi)$ is the scalar by which $f$ acts on the one-dimensional space $\pi^{\K_p}$.

For any $\mu\in X^*(\dmst)$ in the positive Weyl chamber, the image under $\Sat$
of the space of functions in $\Hecke_p$ supported in
$\cup_{\lambda\le\mu}\K_p\lambda^\vee(p)\K_p$ is the span of $e_\lambda$, $\lambda\le\mu$.
Here, $\lambda\le\mu$ means that $\mu-\lambda$ is a sum of simple roots with non-negative coefficients.
It follows that
\begin{multline} \label{eq: suppSat}
\text{for any $h\in\invreg$ there exists $a = a(h)>0$ such that }
\\\supp\Sat_p^{-1}(h)\subset\{x\in G_p:\norm{x}_p\le p^a\}\text{ for all $p$}.
\end{multline}

Denote by $\mu_{\pl,p}$ the Plancherel measure on $\unitdual_p$ with respect to $G_p$.
It is the probability measure characterized by the property
\[
f(e)=\mu_{\pl,p}(\Sat_p f),\ \ f\in\Hecke_p,
\]
or equivalently
\[
\norm{f}_2^2=\mu_{\pl,p}(\abs{\Sat_p f}^2),\ \ f\in\Hecke_p.
\]
It is well known that the support of $\mu_{\pl,p}$ is $\tempdual$.
We will need the following fact.

\begin{lemma} \label{lem: fullsupp}
The support of the weak-$*$ limit of any convergent subsequence of $\mu_{\pl,p}$ is $\tempdual$.
In other words, $\inf_p\mu_{\pl,p}(U)>0$ for any open subset $U\ne\emptyset$ of $\tempdual$.
\end{lemma}

\begin{proof}
The measure $\mu_{\pl,p}$ is absolutely continuous with respect to the Haar measure of $\tempdual$.
The density function is given by Macdonald's formula \cite{MR0435301}*{Chap. 5}.
From this, it easily follows that the weak-$*$ limit of $\mu_{\pl,p}$
as $p\rightarrow\infty$ exists and is also absolutely continuous, with an explicit density function
whose support is $\tempdual$.
\end{proof}

\begin{remark}
In fact, the weak-$*$ limit $\mu_{\operatorname{ST}}$ of $\mu_{\pl,p}$
as $p\rightarrow\infty$, which is often called the Sato--Tate measure,
can be described as follows \cite{MR3437869}*{Proposition 5.3} (an observation going back at least to \cite{MR1018385}).
Let $\hat{K}$ be a maximal compact subgroup of $\hat{G} (\C)$.
The conjugacy classes of $\hat{K}$ are parameterized by $\tempdual / W$.
Let $c: \hat{K} \to \tempdual / W$ be the corresponding conjugation invariant map.
Then $\mu_{\operatorname{ST}}$ is the pushforward of the normalized Haar measure on $\hat{K}$ under
$c$, considered as a $W$-invariant measure on $\tempdual$.
\end{remark}

\subsection{}

For any $M\in\Levis$ we may identify the dual torus of $\mst/T_M$ with a subtorus $\dmstM$ of $\dmst$.
Thus,
\begin{equation} \label{eq: t0M}
\dmstM(\C)=\{\lambda\in\dmst(\C):\chi_\lambda\rest_{T_M (\Q_p)}\equiv 1\}.
\end{equation}
Let $\dmstM(\C)^1=\dmstM(\C)\cap\dmst(\C)^1$.
We will use the following elementary result.
\begin{lemma} \label{lem: elem}
For any $P\in\pars(M)$, $M\ne G$, there exist two disjoint, $W$-invariant, open subsets
$U_1, U_2\ne\emptyset$ of $\tempdual$ and an open neighborhood $U$ of the identity in $\tempdual$ with the following property: for every $\xi\in\tempdual$ the set $\xi U\dmstM(\C)^1$ is disjoint from $U_1$ or $U_2$.
\end{lemma}

\begin{proof}
Fix two $W$-orbits $O_1,O_2$ in $\tempdual$ such that $O_1\dmstM(\C)^1\cap O_2\dmstM(\C)^1=\emptyset$.
Then there exist open, $W$-invariant neighborhoods $U_i$ of $O_i$ in $\tempdual$ such that
$U_1^{\cl}\dmstM(\C)^1$ and $U_2^{\cl}\dmstM(\C)^1$
are disjoint, where $A^{\cl}$ denotes the closure of $A$ in $\tempdual$.
Therefore, every $\xi\in\tempdual$ has an open neighborhood that is
disjoint from $U_1^{\cl}\dmstM(\C)^1$ or from $U_2^{\cl}\dmstM(\C)^1$.
By compactness, there exists an open neighborhood $U$ of the identity in $\tempdual$
such that for any $\xi\in\tempdual$ the neighborhood $\xi U$ of $\xi$ is disjoint from $U_1^{\cl}\dmstM(\C)^1$ or from
$U_2^{\cl}\dmstM(\C)^1$, which is the assertion of the lemma.
\end{proof}

We now state the main technical result of this section.

\begin{proposition} \label{prop: tech}
Let $U\ne\emptyset$ be an open, $W$-invariant subset of $\tempdual$.
Then, there exist constants $A,B,a>0$ and for every $p$, an element $\tau_{U,p}\in(\Hecke_p)_{\sa}$ such that
\begin{enumerate}
\item $\tau_{U,p}(e)=0$.
\item $\norm{\tau_{U,p}}_1\le Bp^A$.
\item $\norm{\tau_{U,p}}_2\le B$.
\item $\Sat_p\tau_{U,p}(x)\ge1$ for all $x\in\hermdual\setminus U$.
\item $\supp\tau_{U,p}\subset\{x\in G_p:\norm{x}_p\le p^a\}$.
\end{enumerate}
\end{proposition}

\begin{remark}
As can be seen from the application in Proposition \ref{prop: mubnd} later,
it is mainly the constant $A$ that influences the quality of our estimates. It would be interesting to find an explicit value for $A$ (for subsets $U_1$ and $U_2$ satisfying the conclusion of Lemma
\ref{lem: elem}).
\end{remark}

Before proving the proposition, we need a few auxiliary results.
The following lemma is necessary to take care of the non-tempered spectrum. For $G = \operatorname{PGL} (n)$ and $G = \operatorname{GL} (n)$ it is contained in \cite{1505.07285}*{Lemma 3.1}.

\begin{lemma} \label{lem: prpr}
There exists $h\in\invreg_{\pd}$ such that $\max_{\tempdual}h=1$ and the restriction of $h$
to $\hermdual$ defines a proper map
\[
h\rest_{\hermdual}:\hermdual\rightarrow\R_{\ge0}.
\]
\end{lemma}

\begin{proof}
The variety $\dmst/W$ is affine and hence by the Noether normalization theorem, it admits
a finite (and in particular, proper) map
\[
f=(p_1,\dots,p_\rk):\dmst/W\rightarrow\mathbf{A}^\rk.
\]
In fact, if $G$ is adjoint then by \cite{MR0240238}*{Th\'eor\`eme VI.3.1}, $f$ may be taken to be an isomorphism and
in general, $f$ can be easily constructed from this case.

Take $h=\sum_{i=1}^rp_ip_i^*$.
Then $h\rest_{\hermdual}=\sum_{i=1}^\rk\abs{p_i}^2$ and therefore,
it is a proper map to $\R_{\ge0}$. Finally,
we may normalize $h$ so that $\max_{\tempdual}h=1$.
\end{proof}

\begin{lemma} \label{lem: aux2}
Let $C$ be a compact, $W$-invariant subset of $\hermdual$ and let $C_1$, $C_2$
be two disjoint, closed, $W$-invariant subsets of $C$.
Then, for every $\epsilon>0$ there exists $f\in\invreg_{\pd}$ such that
\begin{enumerate}
\item $f\rest_C\le1$.
\item $f\rest_{C_1}\le \epsilon$.
\item $f\rest_{C_2}\ge 1-\epsilon$.
\end{enumerate}
\end{lemma}

\begin{proof}
By the Stone-Weierstrass theorem, $\invreg_{\sa}$ is dense in the space of continuous, $W$-invariant, real-valued functions on $C$.
Thus, there exists $h\in\invreg_{\sa}$ such that $h(C)\subset [0,1]$, $h(C_1)\subset[0,\epsilon]$ and $h(C_2)\subset [1-\frac12\epsilon,1]$.
We then take $f=h^2$.
\end{proof}

\begin{lemma} \label{lem: L1norm}
For any $h\in\invreg$ there exist $A,B>0$ such that $\norm{\Sat_p^{-1}h}_1\le Bp^A$ for all $p$.
\end{lemma}

\begin{proof}
Clearly, it is enough to prove the lemma for elements $h$ that form a basis for $\invreg$.
Let $\hat G$ be the dual group of $G$ (with maximal torus $\dmst$).
The restriction to $\dmst$ defines an isomorphism of algebras $\C[\hat G]^{\hat G}\rightarrow\invreg$.
We take the basis $\{h_\lambda:\lambda\in\Lambda_+\}$
formed by the traces of the irreducible rational representations of $\hat G$, indexed by their highest weight.
By the Lusztig--Kato formula (see e.g. \cite{MR2642451})
if $h_\lambda=\Sat_p f_{\lambda,p}$, then $f_{\lambda,p}\ge0$ as a function on $G_p$.
(Up to a power of $p$, the values of $f_{\lambda,p}$ are given by the values of certain Kazhdan--Lusztig polynomials at $p$.)
Hence, $\norm{f_{\lambda,p}}_1=h_\lambda(\rhosymb_p) \le B (\lambda)
p^{\sprod{\lambda}{\rho}}$.
The lemma follows.
\end{proof}

\begin{remark}
The proof yields the value $A = \max_{\lambda \in \Lambda^+: \, c_\lambda \neq 0} \sprod{\lambda}{\rho}$ for $h = \sum_{\lambda \in \Lambda^+} c_\lambda h_\lambda$.
\end{remark}

\begin{proof}[Proof of Proposition \ref{prop: tech}]
We may assume without loss of generality that $U$ consists of regular elements, so that $U$ is open in $\hermdual$.
Let $V\ne\emptyset$ be a $W$-invariant, open subset of $\tempdual$ such $V^{\cl}\subset U$.
By Lemma \ref{lem: fullsupp}, $\delta:=\inf_p\mu_{\pl,p}(V)>0$.
Let $X>0$ be a parameter, to be determined below.
Let $h$ be as in Lemma \ref{lem: prpr} and
let $C$ be a compact, $W$-invariant subset of $\hermdual$, containing $\tempdual$, such that $h>X+2$ outside $C$.
Let $f$ be as in Lemma \ref{lem: aux2} with $C_1=V^{\cl}$, $C_2=C\setminus U$ and $\epsilon=\frac14\delta\le\frac14$.
Clearly,
\[
\mu_{\pl,p}(f)\le 1-(1-\epsilon)\mu_{\pl,p}(V)\le 1-\frac34\delta.
\]
Set
\[
g_p=X(f-\mu_{\pl,p}(f))+h-\mu_{\pl,p}(h)\in\invreg_{\sa}
\]
and $\tau_p=\Sat_p^{-1}(g_p)\in(\Hecke_p)_{\sa}$.
Clearly, $\tau_p(e)=\mu_{\pl,p}(g_p)=0$. On the other hand, for any $x\in\hermdual\setminus U$ we have
\[
g_p(x)\ge\begin{cases}X(\frac34\delta-\epsilon)-1=\frac12X\delta-1,&\text{if }x\in C,\\
-X+X+2 -1=1,&\text{otherwise.}\end{cases}
\]
Thus, taking $X=4\delta^{-1}$ we get $g_p\ge1$ on $\hermdual\setminus U$.
It is also clear that
\[
\norm{\tau_p}_2^2=\mu_{\pl,p}(g_p^2)\le\max_{\tempdual}g_p^2\le (X+1)^2.
\]
Moreover, since $\norm{\tau_p}_1\le \norm{\Sat_p^{-1}(Xf+h)}_1+X+1$ we may apply Lemma \ref{lem: L1norm}
to infer the existence of constants $A,B$ such that $\norm{\tau_p}_1\le Bp^A$ for all $p$.
Finally, the support condition on $\tau_p$ follows immediately from \eqref{eq: suppSat}.
\end{proof}

\section{Archimedean test functions} \label{SectionArch}

Next, we recall the Paley--Wiener theorem for spherical functions, following Harish-Chandra.
See \cites{MR1790156, MR954385} for standard references.
Recall that $G$ denotes a Chevalley group defined over $\Q$ of rank $\rk$.
Let $\Phi$ be the set of roots of $(G,T_0)$ and $\Phi^+$ a fixed
subset of positive roots.
Let $\dm=\dim G-\dim\K_\infty$ be the dimension of the symmetric space
$G (\R) / \K_\infty$. The difference $\dm-\rk$ is the dimension of a maximal unipotent subgroup, i.e., the
number of positive roots of $G$.
The Killing form defines an inner product on $\aaa_0$, and hence on $\aaa_0^*$.
We endow $\aaa_0$ with the Lebesgue measure with respect to the Euclidean structure.

\subsection{} \label{sec: mesinfty}
Let $\Pspace$ denote the space of $W$-invariant Paley--Wiener functions on $\aaa_{0,\C}^*$.
Thus,
\[
\Pspace=\cup_{R>0}\PspaceR{R}
\]
where $\PspaceR{R}$ is the Fr\'echet space of $W$-invariant entire functions $f$ on $\aaa_{0,\C}^*$
such that
\[
\sup_{\lambda\in\aaa_{0,\C}^*}(1+\norm{\lambda})^ne^{-R\norm{\Re\lambda}}\abs{f(\lambda)}<\infty
\]
for all $n>0$.
The space $\Pspace$ is a commutative $*$-algebra under pointwise multiplication and the involution $h^*(\lambda)=\overline{h(-\bar\lambda)}$.
Moreover, the subspaces $\PspaceR{R}$, $R>0$ are invariant under ${}^*$
and $\PspaceR{R_1}\PspaceR{R_2}\subset\PspaceR{R_1+R_2}$ for all $R_1,R_2>0$.

For any $\lambda\in\aaa_{0,\C}^*$, the induced representation
$\Ind_{(P_0)_\infty}^{G_\infty}e^{\sprod{\lambda}{H_0(\cdot)}}$ admits a unique unramified irreducible subquotient
$\pi_\lambda$, which up to equivalence depends only on $W\lambda$.
The map $\lambda\mapsto\pi_\lambda$ defines a homeomorphism between $\aaa_{0,\C}^*/W$ and the unramified
dual of $G_\infty$. We denote by $\prm_\pi$ the $W$-orbit (or a representative thereof) corresponding to
an irreducible unramified representation $\pi$ of $G_\infty$.
Define
\begin{gather*}
\aaa_{0,\herm}^*=\cup_{w\in W}\{\lambda\in\aaa_{0,\C}^*:w\lambda=-\overline{\lambda}\},\\
\iii\aaa_0^*\subset\aaa_{0,\unt}^*=\{\lambda\in\aaa_{0,\C}^*:\pi_\lambda\text{ is unitarizable}\}\subset\aaa_{0,\herm}^*\cap
\{\lambda\in\aaa_{0,\C}^*:\norm{\Re\lambda}\le\norm{\rhosymb}\}.
\end{gather*}
We have $h(\lambda)\in\R$ (resp., $h(\lambda)\ge0$) for all $h\in\Pspace_{\sa}$ (resp., $h\in\Pspace_{\pd}$) and
$\lambda\in\aaa_{0,\herm}^*$.

Denote by $C_c^\infty(\aaa_0)^W$ the $*$-algebra of space of smooth, $W$-invariant, compactly supported functions on $\aaa_0$
under convolution, with $f^*(X)=\overline{f(-X)}$.
For any $R>0$ denote by $C_R^\infty(\aaa_0)^W$ the subspace of $C_c^\infty(\aaa_0)^W$ consisting
of functions supported on the ball $B_R$ of radius $R$ around $0$. It is a Fr\'echet space with respect to the usual topology.
The Fourier--Laplace transform
\[
f\mapsto \int_{\aaa_0}f(X)e^{\sprod\lambda X}\ dX
\]
defines for every $R>0$ an isomorphism of Fr\'echet spaces $C_R^\infty(\aaa_0)^W\rightarrow
\PspaceR{R}$, as well as a $*$-algebra isomorphism $C_c^\infty(\aaa_0)^W\rightarrow\Pspace$.

On the other hand, let
\[
\exp:\aaa_0\rightarrow \mst(\R)
\]
be the exponential map. Thus, the image of $\exp$ is the identity component of $\mst(\R)$ and
$\exp x\chi=\chi(e^x)$ for all $x\in\R$, $\chi\in X_*(\mst)$.
Fix the Haar measure on $U_0(\R)$ as in \cite{MR532745}*{p. 37} and take the Haar measure on $G_\infty$
such that
\[
\int_{G_\infty}f(g)\ dg=\int_{\K_\infty}\int_{U_0(\R)}\int_{\aaa_0}
f(\exp Xuk)\ dX\ du\ dk
\]
where the measure on $\K_\infty$ is normalized by $\vol(\K_\infty)=1$.
Let $\Test$ be the $*$-algebra of smooth, compactly supported, bi-$\K_\infty$-invariant
functions on $G_\infty$ with respect to convolution and with $f^*(g)=\overline{f(g^{-1})}$.
For each $R>0$ let $\TestR$ be the subspace of $\Test$ consisting
of those functions that are supported in $\K_\infty\exp B_R\K_\infty$.
The Harish-Chandra transform
\[
f\mapsto\modulus_0^{\frac12}(\exp X)\int_{U_0(\R)}f(\exp Xu)\ du
\]
defines a $*$-algebra isomorphism $\Test\rightarrow C_c^\infty(\aaa_0)^W$ which for any $R>0$, restricts
to an isomorphism of Fr\'echet spaces $\TestR\rightarrow C_R^\infty(\aaa_0)^W$.
Composing this with the Fourier--Laplace transform, we get a $*$-algebra isomorphism
\[
\sph:\Test\rightarrow\Pspace
\]
which restricts to isomorphisms of Fr\'echet spaces $\TestR\rightarrow\PspaceR{R}$ for all $R>0$.

Let $\beta$ be the Plancherel density. It is given by
\[
\beta(\lambda)=\abs{\cfunc(\lambda)\cfunc(\rhosymb)^{-1}}^{-2},\ \ \lambda\in\iii\aaa_0^*,
\]
where $\cfunc$ is Harish-Chandra's $\cfunc$-function.
More precisely, if
\[
\phi(s)=\frac{\Gamma_{\R}(s)}{\Gamma_{\R}(s+1)},\ \ \Gamma_{\R}(s)=\pi^{-s/2} \Gamma (\frac{s}{2}),
\]
then
\[
\cfunc(\lambda)=\prod_{\alpha \in\Phi^{+}} \phi(\sprod{\lambda}{\alpha^\vee}).
\]
Therefore we have the explicit formula
\[
\beta (\lambda)=\prod_{\alpha\in\Phi^{+}} \left[ \frac{\Gamma (\frac{\langle\rho,\alpha^\vee\rangle}{2})}{\Gamma (\frac{\langle\rho,\alpha^\vee\rangle+1}{2})} \right]^2\,
\prod_{\alpha\in\Phi^{+}}\frac{\langle\lambda,\alpha^\vee\rangle}2
\operatorname{tanh} (\frac{\pi \langle\lambda,\alpha^\vee\rangle}2).
\]

For any $f\in\Test$ we have
\[
f(e)=\int_{\iii\aaa_0^*}\sph f(\lambda)\beta(\lambda)\ d\lambda
\]
where the measure on $\iii\aaa_0^*$ is the dual to the one on $\aaa_0$.

Following Duistermaat--Kolk--Varadarajan \cite{MR532745},
it is useful to introduce the $W$-invariant function
\[
\tilde\beta(\lambda)=\prod_{\alpha\in\Phi^{+}}(1+\abs{\sprod{\lambda}{\alpha^\vee}}),\ \ \lambda\in\iii\aaa_0^*.
\]
We have
\begin{equation} \label{eq: obvs}
\beta(\lambda)\ll\tilde\beta(\lambda)\ll(1+\norm{\lambda})^{\dm-\rk},\ \ \lambda\in\iii\aaa_0^*
\end{equation}
and
\begin{equation} \label{eq: beta12}
\tilde\beta(\lambda_1+\lambda_2)\le\tilde\beta(\lambda_1)\tilde\beta(\lambda_2), \ \ \lambda_1,\lambda_2\in\iii\aaa_0^*.
\end{equation}

\subsection{}
Again following \cite{MR532745}, for any $g\in\Pspace$ and $\mu\in\iii\aaa_0^*$ define $g_\mu \in\Pspace$ by
\[
g_\mu (\lambda) = \frac1{\abs{W}}\sum_{w\in W}g(\lambda-w\mu).
\]
Clearly, $(g_\mu)^*=(g^*)_\mu$ and if $g\in\PspaceR{R}$ then $g_\mu\in\PspaceR{R}$.
In particular, if $g\in\Pspace_{\sa}$ then $g_\mu\in\Pspace_{\sa}$, although it is not
true in general that $g\in\Pspace_{\pd}$ implies $g_\mu\in\Pspace_{\pd}$.

The main feature of $g_\mu$ is that it localizes near $W\mu$. More precisely, for every $n,R>0$ we have
\begin{equation} \label{eq: bndgmu}
\abs{g_\mu(\lambda)}\ll_{n,g,R} (1+\min_{w\in W}\norm{\lambda-w\mu})^{-n}
\ \ \forall\lambda\in\aaa_{0,\C}^*\text{ such that }\norm{\Re\lambda}\le R.
\end{equation}
If $\sph f=g$ then we define $f_\mu=\sph^{-1}(g_\mu)\in\Test$.

\begin{definition} \label{def: special}
\begin{enumerate}
\item
We say that a function $g\in\Pspace$ has a \emph{small derivative} if
\[
\sup_{\norm{\Re\lambda}\le\norm{\rhosymb}}\norm{\nabla g(\lambda)}=
\sup_{\norm{\Re\lambda}\le\norm{\rhosymb}}
\sup_{z\in\aaa_{0,\C}^*:\norm{z}=1}\abs{\partial_zg(\lambda)}\le\frac1{\sqrt{\norm{\rhosymb}^2+1}}.
\]
\item We say that $g\in\Pspace_{\pd}$ (and likewise, $\sph^{-1}g\in\Test_{\pd}$) is \emph{special} if $g(0)=2\abs{W}$
and $g$ has a small derivative.
\end{enumerate}
\end{definition}

It is easy to see that special Paley--Wiener functions exist.
Simply take any $g\in\Pspace_{\pd}$ such that $g(0)=2\abs{W}$ and consider the function $g(t\cdot)$
for $t>0$ sufficiently small.

\begin{lemma}
\begin{enumerate}
\item
If $g$ has a small derivative, then the same is true for $g_\mu$ for any $\mu\in\iii\aaa_0^*$,.
\item Let $g$ be special. Then
\begin{multline} \label{eq: bndspcl}
\text{$\abs{g_\mu(\lambda)}\ge1\ \ \forall\lambda\in\aaa_{0,\C}^*$ such that
$\norm{\Re\lambda}\le\norm{\rhosymb}$ and $\norm{\Im(\lambda-\mu)}\le 1$.}
\end{multline}
\end{enumerate}
\end{lemma}

\begin{proof}
The first part is clear. Suppose that $g$ is special.
Then $g_\mu(\mu)\ge \frac{g(0)}{\abs{W}}\ge2$ and by the mean value theorem,
for any $\lambda\in\aaa_{0,\C}^*$ such that $\norm{\Re\lambda}\le\norm{\rhosymb}$ we have
\begin{multline*}
\abs{g_\mu (\lambda) - g_\mu (\mu)}\le\sqrt{\norm{\Re\lambda}^2 + \norm{\Im(\lambda-\mu)}^2}
\sup_{\norm{\Re\lambda}\le\norm{\rhosymb}}\norm{\nabla g_\mu(\lambda)}
\\\le\sqrt{\frac{\norm{\Re\lambda}^2 + \norm{\Im(\lambda-\mu)}^2}{\norm{\rhosymb}^2+1}}
\end{multline*}
since $g$ has a small derivative. The second part follows.
\end{proof}

\section{Local bounds and Weyl's law for the cuspidal spectrum} \label{sec: Finis-Matz} \label{SectionMainResults}
In this section we will prove the main result of the paper, namely, the Weyl law with remainder for the cuspidal spectrum.
Roughly speaking, the Weyl law is essentially the statement that for suitable test functions,
the main term on the spectral (resp., geometric) side of the trace formula is the contribution of the cuspidal spectrum (resp., of the central elements).
These properties have to be formulated precisely and quantitatively.

We go back to the general setup of Arthur's trace formula in \S \ref{SectionArthur}, except for
our running assumption that $G$ is a simple Chevalley group defined over $\Q$.
We take the product Haar measure on $G(\A)$, where the Haar measures on $G_p$ and $G_\infty$ are as
in \S\ref{sec: padicmes} and \S\ref{sec: mesinfty}, respectively.
For the remainder of the paper we fix a compact open subgroup $K$ of $G(\A_{\fin})$ and let
\[
v_K=\abs{Z(\Q)\cap K}\frac{\vol (G (\Q) \backslash G (\A))}{\vol K}.
\]
Let $\bad=\bad_K$ be the finite set of (finite) primes $p$ such that $\K_p\not\subset K$.
Thus, $K=K_\bad\K^{\bad}$ where $K_\bad=K\cap G(\Q_\bad)$.
Let $\one_K$ be the idemponent corresponding to $K$, i.e., the characteristic function of $K$ normalized by
$(\vol K)^{-1}$, viewed as a smooth function on $G(\Q_{\bad})$.

Consider $\Hecke^S=\otimes_{p\notin S}\Hecke_p$.
For any $f\in\Hecke^S$ and an admissible representation $\pi^S$ of $G(\A^S)$ such that
$\pi^{\K^S}$ is one-dimensional, we denote by $\hat f(\pi^S)$ the scalar
by which $f$ acts on $\pi^{\K^S}$.
We set for any $h\in\Hecke^S$
\begin{equation} \label{def: ms}
\maxsupp(h)=\sup_{x\in G(\A^S):h(x)\ne0}\log\norm{x},
\end{equation}
interpreted as $0$ if $h=0$.

Let $f\in \TestR$, $h\in\Hecke^S$ and $T\in\aaa_0^{\pls}$ and consider $J^T(f\otimes \one_K\otimes h)$ as introduced in \eqref{DefinitionJT}.
By \eqref{eq: polyJ}, we have
\begin{equation} \label{EqnGeometricSpectral}
J^T(f\otimes \one_K\otimes h) =
\mathbb{J}^T(f\otimes \one_K\otimes h)
\quad
\text{for $\regd(T)>C_1(1+\max_{x\in K}\log\norm{x}+R+\maxsupp(h))$.}
\end{equation}

On the geometric side we have the following estimate for
the polynomial function $\mathbb{J}^T(f\otimes \one_K\otimes h)$
\cite{1905.09078}*{Theorem 3.7}.\footnote{In [ibid.], the theorem is stated only for the constant term of the polynomial $\mathbb{J}^T(f\otimes \one_K\otimes h)$. However, the proof yields the full statement of Theorem \ref{thm: finmatz}.}
As in \cite{1905.09078}*{(2.2), (3.1)} set
\[
D ( \lambda) = \min_{M \in \Levis, \, M \neq G}
\prod_{\alpha \in \Phi^+ \bs \Phi^{M,+}} (1 + \abs{\sprod{\lambda}{\alpha^\vee}})^{\frac12},
\]
for classical $G$, and
\[
D ( \lambda) = \frac{1}{\log (2 + \norm{\lambda})} \min_{\substack{M \in \Levis, \\ M \neq G}}
\max_{\substack{S \subset \Phi^+ \bs \Phi^{M,+}, \\ \abs{S} = 2\rk}}
\prod_{\alpha \in S} (1 + \abs{\sprod{\lambda}{\alpha^\vee}})^{\frac12}
\]
for exceptional $G$.
For a qualitative result it is only important that
\begin{equation}
(1 + \norm{\lambda})^{\frac12} \ll D (\lambda) \ll (1 + \norm{\lambda})^{\rk}.
\end{equation}

\begin{theorem}[Finis--Matz] \label{thm: finmatz}
For any $f\in \TestR$, $h\in\Hecke^S$ and $T\in\aaa_0^{\pls}$ we have
\begin{multline*}
\abs{\mathbb{J}^T(f\otimes \one_K\otimes h)-\tilde v_Kh(e)f(e)}\ll_{K, R}\\
\norm{h}_1(1+\norm{T})^\rk \int_{\iii\aaa_0^*}\abs{\sph f(\lambda)}\frac{\beta(\lambda)}{D(\lambda)} \log (2 +  \norm{\lambda})^\rk d\lambda.
\end{multline*}
\end{theorem}

By \eqref{eq: obvs}, \eqref{eq: beta12} and \eqref{eq: bndgmu} we conclude
(cf. \cite{1905.09078}*{Corollary 3.8}):

\begin{corollary} \label{cor: finmatz}
For any $f\in\Test$
and either $F= f_\mu$ or $F = f_\mu*f^*_\mu$
we have
\[
\abs{\mathbb{J}^T(F\otimes\one_K\otimes h)-\tilde v_Kh(e)F(e)}\ll_{f,K}
\norm{h}_1(1+\norm{T})^\rk\log(2+\norm{\mu})^\rk\frac{\tilde\beta(\mu)}{D(\mu)}
\]
for all $h\in\Hecke^S$, $\mu\in\iii\aaa_0^*$ and $T\in\aaa_0^{\pls}$.
\end{corollary}

Turning to the spectral side, for any $M\in\Levis$ denote by $\Pi_2(M)^{\K_\infty K}$ the subset of $\Pi_2(M)$
consisting of the representations such that
$I_P(\pi,0)$ admits a non-zero $\K_\infty K$-invariant vector for any (or equivalently, all) $P\in\pars(M)$.
In particular, if $\pi=\otimes_{p\le\infty}\pi_p\in\Pi_2(M)^{\K_\infty K}$, then $\pi_p$ is unramified
for all $p\notin\bad$ (including $p=\infty$).
As in \S \ref{SectionArthur}, for any $f\in\Test$, $h\in\Hecke^{\bad}$ and $T\in\aaa_0^{\pls}$ we have
\[
J^T(f\otimes\one_K\otimes h)=\sum_{P\in\pars/\sim}J_P^T(f\otimes\one_K\otimes h),
\]
where
\begin{multline} \label{eq: explct}
J^T_P(f\otimes\one_K\otimes h)=\abs{W(M)}^{-1}\\
\sum_{\pi\in\Pi_2(M)^{\K_\infty K}}\int_{\iii\aaa_M^*}\sph f(\prm_{\pi_\infty}+\lambda)
\hat h(I_P(\pi^\bad,\lambda))\sum_{\varphi\in\ortho_P(\pi)^{\K_\infty K}}\norm{\Lambda^TE_P(\cdot,\varphi,\lambda)}_2^2\ d\lambda
\end{multline}
for $P\in\pars(M)$.
Here $\prm_{\pi_\infty}\in(\aaa_0^M)^*_{\C}/W_M$ is the archimedean parameter of $\pi$ and
$\ortho_P(\pi)^{\K_\infty K}$ is an orthonormal basis of the finite-dimensional space $(\AF_{P,\pi}^2)^{\K_\infty K}$
(the $\K_\infty K$-invariant part of $\AF_{P,\pi}^2$).

Let $\nu_P^{T,K}$ be the Radon measure on $\iii\aaa_0^*/W$ given by
\begin{multline*}
\nu_P^{T,K}(f)=\abs{W(M)}^{-1}\\\sum_{\pi\in\Pi_2(M)^{\K_\infty K}}\int_{\iii\aaa_M^*}
f(\Im\prm_{\pi_\infty}+\lambda)
\sum_{\varphi\in\ortho_P(\pi)^{\K_\infty K}}\norm{\Lambda^TE_P(\cdot,\varphi,\lambda)}_2^2\ d\lambda,\ \ f\in C_c(\iii\aaa_0^*)^W,
\end{multline*}
and let
\[
\nu^{T,K}=\sum_{P\in\pars/\sim}\nu^{T,K}_P.
\]
It is clear from \eqref{eq: mntrnc} that
\begin{equation} \label{eq: mntn}
\nu^{T_1,K}\le\nu^{T_2,K}\text{ if }T_1\le T_2.
\end{equation}
We also introduce the Radon measures
\begin{gather*}
\nu_{\cusp}^K=\sum_{\pi\in\Pi_{\cusp}(G)}\dim((\AF_G^{\cusp})_{\pi}^{\K_\infty K})\delta_{\Im\prm_{\pi_\infty}},\\
\nu_{\resid}^{T,K}=\nu_G^{T,K}-\nu_{\cusp}^K,\\
\nu_{\noncusp}^{T,K}=\nu^{T,K}-\nu_{\cusp}^K =
\sum_{P\in\pars, \, P \neq G /\sim}\nu^{T,K}_P + \nu_{\resid}^{T,K}.
\end{gather*}
Let
\[
\mult_{\cusp}^K(A)=\sum_{\pi\in\Pi_{\cusp}(G)^{\K_\infty K}:\prm_{\pi_\infty}\in A}\dim(\AF^G_{\cusp,\pi})^{\K_\infty K}
\]
for any bounded, $W$-invariant subset $A$ of $\aaa_{0,\C}^*$, and
note that
\[
\nu_{\cusp}^K(A)=\mult^K_{\cusp}(A+\aaa_0^*),\ \ A\subset\iii\aaa_0^*.
\]
Finally, we decompose the Radon measure $\nu_{\cusp}^K$ as the sum of two Radon measures
\[
\nu_{\cusp}^K=\nu_{\cusp,\temp}^K+\nu_{\cusp,\nt}^K
\]
(the tempered and non-tempered part of $\nu_{\cusp}^K$) where
\[
\nu_{\cusp,\temp}^K(A)=\mult^K_{\cusp}(A),\ \ \ A\subset\iii\aaa_0^*.
\]

For any $\mu\in\iii\aaa_0^*$ and $R\ge0$ denote by $\ball_R(\mu)$ the ball of radius $R$ around
$\mu$ in $\iii\aaa_0^*$.

The case $h=1$ of Corollary \ref{cor: finmatz} is already sufficient to give the following qualitative local bounds.
Let
\[
\reg(\mu)=\min_{1\ne w\in W}\norm{w\mu-\mu}=\min_{\alpha>0}\abs{\sprod{\mu}{\alpha^\vee}}\norm{\alpha}.
\]

\begin{lemma} \label{lem: bndnt}
For any $\mu\in\iii\aaa_0^*$ and $T\in\aaa_0^{\pls}$ we have
\[
\nu^{T,K}(W\ball_1(\mu))\ll_K(1+\norm{T})^\rk\tilde\beta(\mu).
\]
In particular,
\[
\nu_{\cusp}^K(W\ball_1(\mu))\ll_K\tilde\beta(\mu)
\]
and
\[
\nu^K_{\cusp,\nt}(W\ball_1(\mu))\ll_K\begin{cases}(1+\norm{\mu})^{\dm-\rk-1},&\text{if }\reg(\mu)\le 2,\\
0,&\text{otherwise.}\end{cases}
\]
\end{lemma}

\begin{proof}
Fix $f\in\Test_{\pd}$ special (see Definition \ref{def: special}).
By \eqref{eq: bndspcl} and \eqref{eq: explct} we have
\[
J^T(f_\mu*f_\mu\otimes \one_K)\ge\nu^{T,K}(W\ball_1(\mu))
\]
for any $T\in\aaa_0^{\pls}$.
On the other hand, by Corollary \ref{cor: finmatz} (with $h=1$) and \eqref{EqnGeometricSpectral} we have
\[
\abs{J^T(f_\mu*f_\mu\otimes \one_K)}\ll_{f,K}(1+\norm{T})^\rk\tilde\beta(\mu)
\]
for all $T\in\aaa_0^{\pls}$ with $\regd(T)$ sufficiently large (depending on the support of $f$ and on $K$).
By \eqref{eq: mntn} this holds for all $T\in\aaa_0^{\pls}$.

The last estimate is a consequence of the second one
together with the fact that
$\tilde\beta(\mu)\ll(1+\norm{\mu})^{\dm-\rk-1}$ if $\reg(\mu)\le2$ and
$W\ball_1(\mu)\cap \aaa_{0,\herm}^*=\emptyset$ otherwise.
The lemma follows.
\end{proof}

It is possible to use positivity to deduce a more precise upper bound on the cuspidal spectrum.
To go beyond upper bounds, we will need the full force of Theorem \ref{thm: finmatz} (i.e., for an arbitrary $h$).
First, we make the following elementary, but crucial observation.

\begin{lemma} \label{lem: diric}
For any $M\in\Levis$ there exists an integer $N$ (depending on $K$), divisible by $\prod_{p\in\bad}p$, such that for any
$\pi\in\Pi_2(M)^K$, the central character of $\pi_p$ is trivial on $T_M (\Q_p)$ whenever $p\equiv1\pmod N$.
\end{lemma}

\begin{proof}
This is simply because if $K'$ is an open compact subgroup of $\A_{\fin}^*$, then the characters of $\R_+\Q^*\bs\A^*$
of level $K'$ are precisely the Dirichlet characters of level $N'$ where $N'$ is determined by $K'$.
\end{proof}

The following proposition is the key assertion for the proof of the Weyl law.

\begin{proposition} \label{prop: mubnd}
We have
\[
\nu^{T,K}_{\noncusp}(W\ball_1(\mu))\ll_K\frac{\tilde\beta(\mu)(1+\norm{T})^\rk
\log(2+\norm{\mu})^{3\rk}}{D(\mu)^{(2A+1)^{-1}}}
\ll \frac{\tilde\beta(\mu)(1+\norm{T})^\rk}{(1+\norm{\mu})^\delta}
\]
for all $\mu\in\iii\aaa_0^*$ and $T\in\aaa_0^{\pls}$,
where $A$ is the constant from Proposition \ref{prop: tech} (which depends only on $G$), and
$0 < \delta < (4A+2)^{-1}$.
\end{proposition}

\begin{proof}
First note that it is enough to bound
$\nu_P^{T,K}(W\ball_1(\mu))$ for all proper parabolic subgroups $P$,
as well as $\nu_{\resid}^{T,K}(W\ball_1(\mu))$.

To deal with the first case, fix a proper parabolic subgroup $P\in\pars(M)$ and
let $N$ be as in Lemma \ref{lem: diric}.
By Lemma \ref{lem: elem} there exist two open, $W$-invariant subsets $U_1,U_2\ne\emptyset$ of $\tempdual$ and a number
$\delta_1>0$, such that for every $p$ and $\mu\in\iii\aaa_0^*$ there exists an index $j_p(\mu)\in\{1,2\}$ with
\begin{equation} \label{eq: disjU}
p^{\lambda}\dmstM(\C)\cap U_{j_p(\mu)}=\emptyset\text{ for all }\lambda\in\ball_{\frac{\delta_1}{\log p}}(\mu_P).
\end{equation}
Let $\mu\in\iii\aaa_0^*$ and let $X\ge2$ be a parameter depending on $\mu$ (to be determined below).
Set
\[
\prms_{X,N}=\{p\le X:p\equiv1\pmod N\},\ \ Y=\#\prms_{X,N},
\]
and let $R=\frac{\delta_1}{\log X}$.
For each $p\in\prms_{X,N}$ let $\tau_{U_j,p}\in(\Hecke_p)_{\sa}$, $j=1,2$, be as in Proposition \ref{prop: tech},
and consider the following combination of Hecke operators:
\[
\theta=\theta_{X,\mu,P}=\sum_{p\in\prms_{X,N}}\tau_{U_{j_p(\mu)},p}\in\Hecke^{\bad}_{\sa}.
\]
As a consequence of Proposition \ref{prop: tech}, $\theta$ satisfies the following properties.
\begin{enumerate}
\item For any $\pi\in\Pi_2(M)^K$, $\lambda\in\ball_R(\mu_P)$ and $p\in\prms_{X,N}$ we have
$\widehat{\tau_{U_{j_p(\mu)},p}}(I_P(\pi_p,\lambda))\ge1$ by \eqref{eq: disjU} and Lemma \ref{lem: diric}. Hence,
$\hat\theta(I_P(\pi^\bad,\lambda))\ge Y$.
\item The functions $\tau_{U_{j_p(\mu)},p}$, $p\in\prms_{X,N}$ are pairwise orthogonal in $\Hecke^{\bad}$.
\item $\norm{\theta}_1\le\sum_{p\in\prms_{X,n}}\norm{\tau_{U_{j_p(\mu)},p}}_1\le B X^{A} Y$.
\item $\norm{\theta}_2^2=\sum_{p\in\prms_{X,N}}\norm{\tau_{U_{j_p(\mu)},p}}_2^2 \le B^2 Y$.
\item $\maxsupp(\theta)\ll\log X$.
\end{enumerate}

Fix $f\in\Test_{\pd}$ special and let $F_{\mu,X}=f_\mu*f_\mu\otimes\one_K\otimes\theta*\theta\in\ccg$.
By positivity, the first property of $\theta$ implies that
\[
Y^2 \nu_P^{T,K}(W\ball_R(\mu)) \le J_P^T(F_{\mu,X})
\le J^T(F_{\mu,X})
\]
for all $T\in\aaa_0^{\pls}$.
On the other hand, Corollary \ref{cor: finmatz} and the remaining properties of $\theta$ yield that for $\regd(T)>C_2\log X$, where $C_2$ is independent of $\mu$,
we have
\[
J^T(F_{\mu,X}) = \mathbb{J}^T(F_{\mu,X}) \ll_K B^2 \tilde\beta (\mu)
\left[ Y+\frac{X^{2A} Y^2 (1+\norm{T})^\rk\log(2+\norm{\mu})^\rk}{D(\mu)}\right].
\]
Since $Y \ge c_3 X/\log X$ for all $X \ge C_3$,
where $c_3 > 0$ and $C_3\ge 2$ depend only on $K$, we can take
$X=\max(D(\mu)^{(2A+1)^{-1}}, C_3)$ to get
\[
\nu_P^{T,K}(W\ball_R(\mu))\ll_K
\frac{\tilde\beta(\mu)(1+\norm{T})^\rk
\log(2+\norm{\mu})^{\rk}}{D(\mu)^{(2A+1)^{-1}}}.
\]

Turning to the residual contribution, it follows from Lemma \ref{lem: wallach}
and the fourth property of Proposition \ref{prop: tech},
that for any $\pi\in\Pi_{\resid}(G)$ we have
\[
\hat\theta(\pi^\bad)\ge Y.
\]
Thus, by the same argument as above
\[
\nu_{\resid}^{T,K}(W\ball_R(\mu))\ll_K
\frac{\tilde\beta(\mu)(1+\norm{T})^\rk
\log(2+\norm{\mu})^{\rk}}{D(\mu)^{(2A+1)^{-1}}}.
\]
Since $R\ge\frac{\delta_2}{\log(1+\norm{\mu})}$ for a suitable $\delta_2>0$,
we obtain
\[
\nu_{\noncusp}^{T,K}(W\ball_1(\mu))\ll_K
\frac{\tilde\beta(\mu)(1+\norm{T})^\rk
\log(2+\norm{\mu})^{2\rk}}{D(\mu)^{(2A+1)^{-1}}}
\]
at least under the condition $\regd(T)>C_2\log X \gg \log (2 + \norm{\mu})$.
However, using \eqref{eq: mntn} we can remove this condition (at the cost of replacing the exponent
$2\rk$ of $\log(2+\norm{\mu})$ by $3\rk$).
\end{proof}

\begin{theorem} \label{thm: locbnd}
There exists $\delta>0$ (depending only on $G$) such that for any $f\in\Test$ and $\mu\in\iii\aaa_0^*$ we have
\[
\abs{\tr R_{\cusp}(f_\mu\otimes \one_K)- v_K f_\mu(e)}\ll_{f,K}
\frac{\tilde\beta(\mu)}{(1+\norm{\mu})^\delta}.
\]
\end{theorem}

\begin{proof}
For any $g\in\Pspace$ let $\mxml g\in C(\iii\aaa_0^*)^W$ be given by
\[
\mxml g(\lambda)=\max_{x\in\aaa_0^*:\norm{x}\le\norm{\rhosymb}}\abs{g(\lambda+x)},\ \ \lambda\in\iii\aaa_0^*.
\]
Then, for any $f\in\Test$ we have
\[
\abs{\tr R_{\cusp}(f\otimes\one_K)-J^T(f\otimes\one_K)}\le\nu_{\noncusp}^{K,T}(\mxml(\sph f)).
\]
Thus, using \eqref{EqnGeometricSpectral} and Corollary \ref{cor: finmatz}, in order to finish the proof it remains to prove that
\begin{equation} \label{UpperboundNuNoncusp}
\nu_{\noncusp}^{K,T}(\mxml(\sph f_\mu))\ll_{f,K}\frac{\tilde\beta(\mu)(1+\norm{T})^\rk}{(1+\norm{\mu})^\delta}
\end{equation}
for all $T\in\aaa_0^{\pls}$. Clearly,
\[
\nu_{\noncusp}^{K,T}(\mxml(\sph f_\mu))\le\sum_{k=1}^\infty \nu_{\noncusp}^{K,T}(W\ball_k(\mu))\sup_{(W\ball_{k-1}(\mu))^c}\mxml(\sph f_\mu).
\]
By \eqref{eq: bndgmu}, for every $n>0$ we have
\[
\sup_{(W\ball_{k-1}(\mu))^c}\mxml(\sph f_\mu)\ll_{f,n}k^{-n}.
\]
Covering $\ball_k(\mu)$ by $O(k^\rk)$ balls of radius $1$ with centers in $\ball_k(\mu)$ and using
either Proposition \ref{prop: mubnd} and \eqref{eq: beta12} if $k\le\frac12\norm{\mu}$ or Lemma \ref{lem: bndnt} and \eqref{eq: obvs} if $k>\frac12\norm{\mu}$, we get
\[
\nu_{\noncusp}^{K,T}(W\ball_k(\mu))\ll_Kk^\dm(1+\norm{T})^\rk\cdot
\begin{cases}\frac{\tilde\beta(\mu)}{(1+\norm{\mu})^\delta}&\text{if }k\le\frac12\norm{\mu},\\
1&\text{otherwise.}\end{cases}
\]
Our claim follows.
\end{proof}

\begin{corollary} \label{cor: lclbnd}
For any $f\in\Test$ and $\mu\in\iii\aaa_0^*$ we have
\[
\abs{\nu_{\cusp,\temp}^K(\sph f_\mu)- v_K f_\mu(e)}\ll_{f,K}(1+\norm{\mu})^{\dm-\rk-\delta}.
\]
\end{corollary}

\begin{proof}
Indeed,
\[
\abs{\tr R_{\cusp}(f_\mu\otimes \one_K)-\nu_{\cusp,\temp}^K(\sph f_\mu)}\le \nu_{\cusp,\nt}^K(\abs{\sph f_\mu}).
\]
On the other hand, by \eqref{eq: bndgmu}, for all $n>0$ we have
\[
\nu_{\cusp,\nt}^K(\abs{\sph f_\mu})\ll_{f,n}\sum_{k=1}^\infty\nu_{\cusp,\nt}^K(W\ball_k(\mu))k^{-n},
\]
and by Lemma \ref{lem: bndnt}
\[
\nu_{\cusp,\nt}^K(W\ball_k(\mu))\ll k^\rk(k+\norm{\mu})^{\dm-\rk-1}.
\]
Hence, the corollary follows from Theorem \ref{thm: locbnd}.
\end{proof}

\begin{definition}
For any $A\subset\iii\aaa_0^*$ and $R>0$ let
\[
\bdry_RA=\{x\in\iii\aaa_0^*:\ball_R(x)\cap A\ne\emptyset\text{ and }\ball_R(x)\not\subset A\}.
\]
\end{definition}
Note that for any $x\in\bdry_RA$ we have $\ball_R(x)\subset\bdry_{2R}A$.
Hence, by the Vitali covering lemma, $\bdry_{\frac12}A$ is covered by $O(\vol\bdry_1A)$ balls of radius $1$.
On the other hand, it is clear that if $x\in\bdry_RA$, then $\dist(x,\bdry_\epsilon A)\le R$ for any $\epsilon>0$.
(Here $\dist(x,B)=\inf_{b\in B}\norm{x-b}$.)
It follows that for every $R\ge1$,
\begin{equation} \label{eq: cvr}
\bdry_RA\text{ is covered by $O(R^\rk\vol\bdry_1A)$ balls of radius $1$.}
\end{equation}
In particular, $\vol(\bdry_RA)\ll R^\rk\vol(\bdry_1A)$.

Let $D\subset\iii\aaa_0^*$ be a $W$-invariant bounded measurable set.
For any $f\in\Test$ define $f_D\in\Test$ by
\[
f_D(g)=\int_Df_\mu(g)\ d\mu, \ \ g\in G_\infty.
\]
Similarly, for any $g\in\Pspace$ let $g_D\in\Pspace$ given by
\[
g_D(\lambda)=\int_Dg_\mu(\lambda)\ d\mu=\int_Dg(\lambda-\mu)\ d\mu,\ \ \lambda\in\aaa_{0,\C}^*.
\]
Thus, if $\sph f=g$ then $\sph f_D=g_D$.

\begin{lemma} \label{lem: appid}
Let $D\subset\iii\aaa_0^*$ be a $W$-invariant bounded measurable set and $\chi_D$ its characteristic function.
Let $f\in\Pspace$ with $\int_{\iii\aaa_0^*}f(\lambda)\ d\lambda=1$. Then for any $n>0$ and $\lambda\in\iii\aaa_0^*$ we have
\[
\abs{f_D(\lambda)-\chi_D(\lambda)}\ll_{f,n}\begin{cases}(1+\dist(\lambda,D^c))^{-n},&\lambda\in D,\\
(1+\dist(\lambda,D))^{-n},&\text{otherwise.}\end{cases}
\]
\end{lemma}

This follows immediately from the rapid decay of $f$ on $\iii\aaa_0^*$.

\begin{corollary} \label{cor: appid}
Let $\nu$ be a measure on $\iii\aaa_0^*/W$ satisfying
\[
\nu(W\ball_1(\lambda))\ll(1+\norm{\lambda})^m,\ \ \ \lambda\in\iii\aaa_0^*,
\]
for some $m\ge0$.
Then for any $W$-invariant, bounded measurable set $D$ in $\iii\aaa_0^*$ and any $f\in\Pspace$ such that $\int_{\iii\aaa_0^*}f(\lambda)\ d\lambda=1$ we have
\[
\abs{\nu(D)-\nu(f_D)}\ll_f\vol(\bdry_1D)(1+\norm{D})^m,
\]
where $\norm{D}=\sup_{\lambda\in D}\norm{\lambda}$.
\end{corollary}

\begin{proof}
By Lemma \ref{lem: appid}, for every $n>0$ we have
\[
\abs{\nu(D)-\nu(f_D)}\ll_{f,n}\sum_{k=1}^\infty\nu(\bdry_kD)k^{-n}.
\]
On the other hand, by \eqref{eq: cvr}, for any $k\ge1$, $\bdry_kD$ is covered by $O(k^\rk\vol(\bdry_1D))$ balls of radius $1$
and we may as well assume that the centers of these balls lie in $\ball(0,\norm{D}+k)$.
Hence, by the assumption on $\nu$ we have
\[
\sum_{k=1}^\infty\nu(\bdry_kD)k^{-n}\ll\vol(\bdry_1D)(1+\norm{D})^m
\]
provided that $n\ge \rk+m+2$. The corollary follows.
\end{proof}

\begin{theorem} \label{TheoremWeyl}
There exists $\delta>0$ such that for any open subgroup $K$ of $G(\A_{\fin})$ and any $W$-invariant,
bounded measurable set $D$ in $\iii\aaa_0^*$ we have
\[
\abs{\mult_{\cusp}^K(D)-v_K\mu_{\pl}(D)}\ll_K\vol(\bdry_1D)(1+\norm{D})^{\dm-\rk}+(1+\norm{D})^{\dm-\delta}.
\]
In particular, if the boundary of $D$ is rectifiable (or more generally, if the $(\rk-1)$-dimensional
upper Minkowski content of $D$ is finite), then for all $t\ge1$ we have
\[
\abs{\mult_{\cusp}^K(t D)-v_K\mu_{\pl}(tD)}\ll_{K, D} t^{\dm-\delta}.
\]
\end{theorem}

Note that $\mu_{\pl}(tD) = C_D t^{\dm} + O (t^{\dm-1})$ for $t \to \infty$.

\begin{proof}
Fix $f\in\Pspace$ such that $\int_{\iii\aaa_0^*}f(\lambda)\ d\lambda=1$.
By Corollary \ref{cor: lclbnd}
\[
\abs{\nu_{\cusp,\temp}^K(f_{\mu})-v_K \mu_{\pl}(f_{\mu})}\ll_{f,K} (1+\norm{\mu})^{\dm-\rk-\delta}.
\]
Integrating over $\mu\in D$ we get
\[
\abs{\nu^K_{\cusp,\temp}(f_D)-v_K \mu_{\pl}(f_D)}\ll_{f,K} (1+\norm{D})^{\dm-\delta}.
\]
On the other hand, by Corollary \ref{cor: appid} we have
\[
\abs{\nu_{\cusp,\temp}^K(f_D)-\mult_{\cusp}^K(D)}\ll_{f,K} \vol(\bdry_1D)(1+\norm{D})^{\dm-\rk}
\]
and
\[
\abs{\mu_{\pl}(f_D)-\mu_{\pl}(D)}\ll_f\vol(\bdry_1D)(1+\norm{D})^{\dm-\rk}.
\]
The theorem follows.
\end{proof}

For the balls $\ball_t(0)$, $t \ge 1$, we obtain:

\begin{corollary} \label{cor: stdweylaw}
\begin{equation} \label{EqnWeylLaw}
\mult_{\cusp}^K(\ball_t(0)) = \frac{v_K t^\dm}{(4 \pi)^{\dm/2}\Gamma(\frac{\dm}{2}+1)} + O_K (t^{\dm-\delta})
\end{equation}
and
\begin{equation} \label{EqnWeylLaw2}
\mult_{\cusp}^K(\{ \lambda \in \aaa_{0,\C}^* :
\norm{\Im \lambda}^2 - \norm{\Re \lambda}^2 + \norm{\rhosymb}^2 \le t^2 \})
= \frac{v_K t^\dm}{(4 \pi)^{\dm/2}\Gamma(\frac{\dm}{2}+1)} + O_K
(t^{\dm-\delta}).
\end{equation}
Thus, the Weyl law with remainder holds for the cuspidal spectrum of the adelic quotient $G (\Q) \bs G (\A) / \K_\infty K$.
If in addition $G$ is simply connected, then we obtain
\[
N_{\Gamma \bs G (\R) / \K_\infty,\cusp}(t) = \frac{\vol (\Gamma \bs G (\R) / \K_\infty) t^\dm}{(4\pi)^{\dm/2} \Gamma (\frac{\dm}{2} + 1)} + O_K (t^{\dm-\delta})
\]
for $\Gamma = G (\Q) \cap K$
by strong approximation (cf. Theorem \ref{thm: basicweylaw}).
\end{corollary}

Indeed, the first statement follows from the computation of $\mu_{\pl}(\ball_t(0))$
(cf. \cite{MR532745}*{\S8}).
The second statement follows from the first one since $\norm{\Re\lambda}\le\norm{\rhosymb}$
for any $\lambda\in\aaa_{0,\unt}^*$.
Finally, note that the Laplace eigenvalue corresponding to an archimedean parameter
$\lambda$ is $\norm{\Im \lambda}^2 - \norm{\Re \lambda}^2 + \norm{\rhosymb}^2$.

Finally, we give a generalization of Theorem \ref{TheoremWeyl} incorporating Hecke operators
$\tau \in \Hecke^{\bad}$.
Set
\[
\mult_{\cusp}^K(A, \tau)=\sum_{\pi\in\Pi_{\cusp}(G)^{\K_\infty K}:\prm_{\pi_\infty}\in A}\tr \left. \tau \right|_{(\AF^G_{\cusp,\pi})^{\K_\infty K}}
\]
for any bounded, $W$-invariant subset $A$ of $\aaa_{0,\C}^*$.

\begin{theorem} \label{TheoremWeylHecke}
There exists $\delta>0$ such that for any open subgroup $K$ of $G(\A_{\fin})$,
any $\tau \in \Hecke^{\bad}$ and any $W$-invariant,
bounded measurable set $D$ in $\iii\aaa_0^*$ we have
\[
\abs{\mult_{\cusp}^K(D, \tau)-v_K \tau (e) \mu_{\pl}(D)}\ll_K
\norm{\tau}_1 \big(\vol(\bdry_1D)(1+\norm{D})^{\dm-\rk}+
(1 + \maxsupp (\tau))^{\rk} (1+\norm{D})^{\dm-\delta}\big).
\]
In particular, if the $(\rk-1)$-dimensional upper Minkowski content of $D$ is finite,
(e.g., if the boundary of $D$ is rectifiable), then for all $t\ge1$ we have
\[
\abs{\mult_{\cusp}^K(t D, \tau)-v_K  \tau (e) \mu_{\pl}(tD)}
\ll_{K, D} \norm{\tau}_1 (1+ \maxsupp (\tau))^{\rk} t^{\dm-\delta}.
\]
\end{theorem}

\begin{proof}
We first prove the analog of Theorem \ref{thm: locbnd}: for any $f\in\Test$ and $\mu\in\iii\aaa_0^*$ we have
\begin{equation} \label{RcuspwithHecke}
\abs{\tr R_{\cusp}(f_\mu\otimes \one_K \otimes\tau)-v_K \tau (e) f_\mu(e)}\ll_{f,K}
\norm{\tau}_1 (1 + \maxsupp (\tau))^{\rk}
\frac{\tilde\beta(\mu)}{(1+\norm{\mu})^\delta}.
\end{equation}
As in the proof of Theorem \ref{thm: locbnd}, we start with
\[
\abs{\tr R_{\cusp}(f_\mu \otimes\one_K\otimes\tau)-J^T(f_\mu\otimes\one_K\otimes\tau)}\le\norm{\tau}_1 \nu_{\noncusp}^{K,T}(\mxml(\sph f_\mu))
\]
for all $\tau \in \Hecke^{\bad}$ and $T\in\aaa_0^{\pls}$. Moreover, by \eqref{EqnGeometricSpectral} we have
\[
J^T(f_\mu\otimes\one_K\otimes\tau) = \mathbb{J}^T(f_\mu\otimes\one_K\otimes\tau)
\]
for $\regd(T) \gg_{f,K} 1 + \maxsupp (\tau)$. Using the upper bound
of \eqref{UpperboundNuNoncusp} for
$\nu_{\noncusp}^{K,T}(\mxml(\sph f_\mu))$ and the estimate of Corollary \ref{cor: finmatz}
for $\mathbb{J}^T(f_\mu\otimes\one_K\otimes\tau)$
for a suitable value of $T$, one obtains \eqref{RcuspwithHecke}.

Denote by $R_{\cusp,\temp}$ the restriction of $R_{\cusp}$ to the space of
cuspidal representations tempered at infinity.
As in the proof of Corollary \ref{cor: lclbnd}, one obtains
\[
\abs{\tr R_{\cusp} (f_{\mu} \otimes\one_K\otimes\tau) -
\tr R_{\cusp,\temp} (f_{\mu} \otimes\one_K\otimes\tau)}\ll_{f,K}
\norm{\tau}_1 (1+\norm{\mu})^{\dm-\rk-1},
\]
and if in addition $\int_{\iii\aaa_0^*}\sph f(\lambda)\ d\lambda=1$, then by Corollary \ref{cor: appid}
\[
\abs{\tr R_{\cusp,\temp} (f_{D} \otimes\one_K\otimes\tau) - \mult_{\cusp}^K(D, \tau)}\ll_{f,K} \norm{\tau}_1 \vol(\bdry_1D)(1+\norm{D})^{\dm-\rk}.
\]
From these estimates one can deduce the theorem exactly as in the proof of
Theorem \ref{TheoremWeyl}.
\end{proof}

%\bibliographystyle{amsalpha}
%\bibliography{../Bibfiles/all}
%\end{document}

\def\cprime{$'$}
% \bib, bibdiv, biblist are defined by the amsrefs package.

\end{document}